\documentclass[12pt]{amsart}
\usepackage{amsmath,amssymb,amsthm}
\usepackage{amscd}
\usepackage{mathrsfs}

\usepackage{comment}
\usepackage{graphicx}
\usepackage{color}
\usepackage[abs]{overpic} 
\usepackage{ulem} 
\usepackage[dvipdfm,left=25truemm,right=25truemm,top=30truemm,bottom=30truemm]{geometry}

\usepackage{bm}

\theoremstyle{plain} 
\newtheorem{theorem}{Theorem}[section] 
\newtheorem{lemma}[theorem]{Lemma}

\newtheorem{proposition}[theorem]{Proposition}

\theoremstyle{definition}
\newtheorem{definition}[theorem]{Definition}
\newtheorem{remark}[theorem]{Remark}

\makeatletter
\def\Bline{%
\noalign{\ifnum0=`}\fi\hrule \@height 1pt \futurelet
\reserved@a\@xhline}
\@addtoreset{equation}{section}

\makeatother

\allowdisplaybreaks[3]

\newcommand{\textR}[1]{#1}

\newcommand{\R}{\mathbb{R}}

\newcommand{\pd}[2]{\dfrac{\partial#1}{\partial#2}}

\newcommand{\vp}{\mathbf{p}}

\renewcommand{\phi}{\varphi}
\renewcommand{\epsilon}{\varepsilon}

\newcommand{\wtilde}[1]{\widetilde{#1}}

\renewcommand{\geq}{\geqslant}
\renewcommand{\leq}{\leqslant}

\newcommand{\dsum}[2]{\displaystyle \sum_{#1}^{#2}}

\newcommand{\ma}{\mathcal{A}}
\newcommand{\wt}[1]{\widetilde{#1}}
\newcommand{\tk}[1]{\widetilde{\kappa}_{#1}}
\newcommand{\tv}[1]{\widetilde{\bm{v}}_{#1}}

\begin{document}

\title[Distance squared functions on singular surfaces with $S_k$,
$B_k$, $C_k$ and $F_4$ singularities]
{Distance squared functions on singular surfaces parameterized by
smooth maps $\mathcal{A}$-equivalent to $S_k$, $B_k$, $C_k$ and $F_4$.}

\author[T.~Fukui and M.~Hasegawa]
{Toshizumi Fukui and Masaru Hasegawa}

\address[Toshizumi Fukui]{%
Departmet of Mathematics, Faculty of Science, Saitama University, 
Saitama, 338-8570, JAPAN.
}
\email{tfuku@rimath.saitama-u.ac.jp}

\address[Masaru Hasegawa]{%
Depart of Information Science,
Center for Liberal Arts and Science,
Iwate Medical University,
Iwate 028-3694, JAPAN.
}
\email{mhase@iwate-med.ac.jp}



\thanks{
The first author is partially supported by grant-in-Aid in Science
26287011.
The second author is supported by the FAPESP post-doctoral grant number
2013/02543-1 during a post-doctoral period at ICMC-USP} 

\textR{\begin{abstract}
We describe singularities of distance squared functions on singular surfaces in $\R^3$ parameterized by smooth map-germs $\mathcal{A}$-equivalent to one of $S_k$, $B_k$, $C_k$ and $F_4$ singularities in terms of extended geometric language via finite succession of blowing-ups. 
We investigate singularities of wave-fronts and caustics of such singular surfaces.
\end{abstract}}

\maketitle

\section{Introduction}

Let $f:(\R^2,0)\to(\R^3,0)$ be a smooth map-germ which locally defines a surface $S$ (possibly with singularities) in $\R^3$.
We consider the family $D:(\R^2,0)\times\R^3\to\R$ of functions on $S$ defined by 
\[
D(u,v,\vp) = \dfrac12\|f(u,v)-\vp\|^2,
\]
where $\vp \in \R^3$. 
We define $d_{\vp}(u,v)=D(u,v,\vp)$, which is the distance squared function on $S$ from the point $\vp$. 
In principle, this function measures contact of $S$ with spheres centered at $\vp$, and the family $D$ is $3$-parameter unfolding of $d_{\vp}$. 
We have investigated when $D$ is $\mathcal{K}$ and $\mathcal{R}^+$-versal for a regular surface \cite{FH2012-2} or a singular surface with a Whitney umbrella (cross cap) \cite{FH2012-1}. 
It is important to study the $\mathcal{K}$ and $\mathcal{R}^+$-versality of $D$, since the number of parameters of the unfolding determines diffeomorphism type of $\mathcal{K}$ and $\mathcal{R}^+$-versal unfoldings and thus diffeomorphism type of the discriminant set and bifurcation set of $D$. 
Since the discriminant sets of $D$ are wave-fronts of $S$ and the bifurcation sets of $D$ are caustics of $S$, this enables us to determine the diffeomorphism typs of the wave-fronts and the caustics. 
Next target is to generalize these results for a singular surface which is the image of an $\mathcal{A}$-simple map-germ. 
$\mathcal{A}$-simple map-germs are classified by Mond \cite{Mond1985} and the list of the classification is given in Table \ref{tab:A-simple}. 
In this paper, we investigate when $D$ is $\mathcal{K}$ and $\mathcal{R}^+$-versal for such singular surfaces except $H_k$.

\begin{table}[ht!]
\caption{Classes of $\mathcal{A}$-simple map-germs.}
\centering
\begin{tabular}{ccc}
\Bline
 Name & Normal form & $\mathcal{A}$-codim.\\\hline
 Immersion & $(x, y, 0)$ & $0$ \\
 Whitney umbrella ($S_0$) & $(x, y^2, x y)$ & $2$ \\
 $S_k^\pm$ & $(x, y^2, y^3 \pm x^{k+1} y)$,\, $k\geq1$ & $k+2$\\
 $B_k^\pm$ & $(x, y^2, x^2 y \pm y^{2k+1})$,\, $k\geq2$ & $k+2$\\
 $C_k^\pm$ & $(x, y^2, x y^3 \pm x^k y)$,\, $k\geq3$ & $k+2$\\
 $F_4$ & $(x, y^2, x^3 y + y^5)$ & $6$\\
 $H_k$ & $(x, x y + y^{3k-1}, y^3)$,\, $k\geq2$ & $k+2$\\\Bline
\end{tabular}\\
  (When $k$ is even, $S_k^+$ is equivalent to $S_k^-$, and $C_k^+$ to $C_k^-$.)
\label{tab:A-simple}
\end{table}

Our main result (Theorem 3.1) is to describe singularities of the distance-squared functions on our singular surfaces and conditions for their unfoldings being $\mathcal{K}$ and $\mathcal{R}^+$-versal in terms of differential geometry of the singular surfaces.  
As a consequence, we obtain criteria (Theorem 4.3) for types of singularities of wave-fronts and caustics of our singular surfaces. 
To do this, we introduce differential geometric language for such singular surfaces using finite succession of blowing-ups. 
The notions we introduce are not enough to describe differential geometry for a singular surface with $H_k$ and we think it is better to treat $H_k$ case separately. 
We plan to prepare another article for $H_k$ case. 
 
The paper is organized as follows.
In Section 2, we investigate the differential geometric information extended to singularities of our singular surfaces. 
In Section 3, we show the criteria for singularities of wave-fronts and caustics of the singular surfaces and investigate distance squared functions on the singular surfaces. 
In addition, we introduce focal loci which should be considered as analogy of focal conics of Whitney umbrellas. 
In Appendix A, we collect closed formulas for coefficients of differential geometric ingredients defined in Section 2, since these are often not short. 

\section{Differential geometry for singular surfaces}

Whitney \cite{Whitney} showed that smooth maps of $\R^2$ into $\R^3$ can have singularities which are not avoidable by small perturbation.
Such a singularity is called a Whitney umbrella or cross-cap (Figure \ref{fig:WH}).
Since Whitney umbrellas are stable singularities, it is natural to seek their geometry.
The extrinsic differential geometry of the Whitney umbrella is
investigated in \cite{BW1998, FH2012-1, FH2013, GGS2000, GS1986, NT2007, Oliver2011, Tari2007, West1995}, and in \cite{HHNSUY2014, HHNUY2012} its intrinsic properties are considered. 

\begin{figure}[ht!]
\centering
\includegraphics[width=0.25\textwidth,clip]{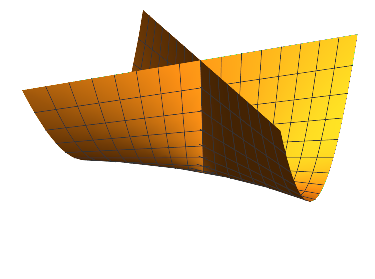}
\caption{The Whitney umbrella, $(u,v)\to(u, u v, v^2)$}
\label{fig:WH}
\end{figure}

Two map-germs $f, g:(\R^2,0)\to(\R^3,0)$ are said to be $\mathcal{A}$-equivalent if $g = \Phi \circ f \circ \phi^{-1}$ for some germs of diffeomorphisms $\phi$ and $\Phi$ of, respectively, the source and target. 
In \cite{Mond1985}, Mond classified smooth map-germs $(\R^2,0)\to(\R^3,0)$ under $\mathcal{A}$-equivalence and gave a list (Table \ref{tab:A-simple}) of normal forms of the map-germs.
Several authors tried to research in this direction, see for example \cite{MN-B2015, MN-B2009, O-ST2012, Saji2010}.

\subsection{Normal forms of corank 1 singularities.}
To analyze the differential geometry of a surface, relevant parameterizations of the surface are essential. 
However, we can not use the normal forms given in Table \ref{tab:A-simple} as its parameterization  because local differential geometry of the surfaces may not be preserved by diffeomorphisms in the target.
So we construct a parameterizasion by using changes of coordinates in the source and isometries in the target, which preserve the geometry of the surface. 

Given a map-germ $(\R^2,0)\to(\R^3,0)$ of corank 1 at the origin, we can make a change of coordinates in the source and a rotation in the target and write the germ in the form
\[
(u,v)\mapsto(u, y(u,v), z(u,v)),
\]
where $y$, $z\in\langle u,v \rangle_{\mathcal{E}_2}^2$.
Here, $\mathcal{E}_2$ is the local ring of smooth function germs of $(\R^2,0)\to\R$.

\begin{proposition}
\label{prop:normal_form}
Let $g:(\R^2,0)\to(\R^3,0)$ be a map-germ of corank 1 at the origin. 
Then, after using rotations in the target and changes of coordinates in  the source, we can reduce $g$ to the form
\begin{equation}
\label{eq:normal_form_corank1}
\left(u,\, \frac12 v^2 + \sum_{i=2}^k \frac{b_i}{i!} u^i + O(u,v)^{k+1},\, \frac12 a_{2,0} u^2 + \sum_{m=3}^k\sum_{i+j=m} \frac{a_{i,j}}{i!j!}u^i v^j + O(u,v)^{k+1}\right),
\end{equation}
if $j^2g(0)$ is $\ma$-equivalent to $(u,v^2,0)$, or
\begin{equation}
\label{eq:normal_form_corank1-2}
\left(u,\, u v + \sum_{i=3}^k \frac{b_i}{i!} v^i + O(u,v)^{k+1},\,  \frac12 a_{2,0} u^2 + \sum_{m=3}^k\sum_{i+j=m}\frac{a_{i,j}}{i!j!} u^i v^j + O(u,v)^{k+1}\right), 
\end{equation}
if $j^2g(0)$ is $\ma$-equivalent to $(u,u v,0)$.
\end{proposition}

\begin{proof}
We may assume that
\[
 j^2g(0) =
\left(u,\, \frac12 b_{2,0} u^2 + b_{1,1} u v + \frac12 b_{0,2} v^2,\, \frac12 a_{2,0} u^2 + a_{1,1} u v + \frac12 a_{0,2} v^2 \right).
\]
If $j^2g(0)$ is $\mathcal{A}$-equivalent to $(u,v^2,0)$, we can assume that $(a_{0,2},b_{0,2})\ne(0,0)$ and
\[
\begin{vmatrix}
  b_{1,1} & b_{0,2}\\
  a_{1,1} & a_{0,2}\\
\end{vmatrix} = 0.
\]
Let $R$ be the orthogonal matrix defined by 
\[
R=
\begin{pmatrix}
1 & 0 \\
0 & R_1
\end{pmatrix}
\quad\text{where}\quad R_1 = \frac1{\sqrt{a_{0,2}^2 + b_{0,2}^2}}
\begin{pmatrix}
b_{0,2} & a_{0,2}\\
-a_{0,2} & b_{0,2}
\end{pmatrix}.
\]
Then the 2-jet of $Rg$ is 
\[
\left(u,\, \frac{a_{2,0}\,b_{0,2} + a_{0,2}\,b_{2,0}}{2\sqrt{a_{0,2}^2 + b_{0,2}^2}} u^2 + \frac{a_{0,2}\,a_{1,1} + b_{2,0}\,b_{1,1}}{\sqrt{a_{0,2}^2 + b_{0,2}^2}} uv + \frac{\sqrt{a_{0,2}^2 + b_{0,2}^2}}2 v^2, \, \frac{a_{2,0}\,b_{0,2} - a_{0,2}\,b_{2,0}}{2\sqrt{a_{0,2}^2 + b_{0,2}^2}} u^2 \right).
\]
Substituting $v$ by $c_{1,0} u + c_{0,1} v$ and choosing suitable coefficients $c_{1,0}$ and $c_{0,1}$,
we show that the 2-jet of $Rg$ is
\[
\left(u,\, \frac12 b_{2,0}^* u^2 + \frac12 v^2,\, \frac{a_{2,0}\,b_{0,2} - a_{0,2}\,b_{2,0}}{2\sqrt{a_{0,2}^2 + b_{0,2}^2}} u^2 \right).
\]
This shows the first assertion for $k=2$.
We proceed by induction of $k$.
Assume that $g$ is in the form \eqref{eq:normal_form_corank1}.
Substituting $v$ by $v + \sum_{i+j=k} c_{i,j} u^i v^j/(i!j!)$, the second component of $g$ is 
\[
\frac12 v^2 + \frac{b_{k+1,0}}{(k+1)!} u^{k+1} + \sum_{i+j=k}
\left(\frac{b_{i-1,j}}{(i-1)!j!} + \frac{c_{i,j}}{i!j!}\right) u^{i+1} v^j + O(u,v)^{k+2}
\]
and we can choose $c_{i,j}$ so that the term $u^iv^j$ ($i+j=k+1$, $i\geq 1$) are zero, which conclude the first assertion.
We skip the proof of the second assertion, because the proof is similar to that of the first assertion. 
\end{proof}

\begin{proposition}
\label{prop:CriteriaOfA}
Necessary and sufficient conditions for $g$ given in \eqref{eq:normal_form_corank1} to be $\ma$-equivalent to one of $S_k$, $B_k$, $C_k$, and $F_4$ are as follows: 
\begin{align*}
S_1:&\quad \uwave{a_{2,1}\ne0},\ a_{0,3} \ne0,\\
S_{k\geq2}:&\quad \uwave{a_{2,1} = \cdots = a_{k,1} = 0,\
a_{k+1,1}\ne0},\ a_{0,3}\ne0,\\ 
B_2:&\quad a_{0,3} = 0,\ \uwave{a_{2,1} \ne0},\ 3a_{0,5}\,a_{2,1} - 5a_{1,3}^2 \ne0,\\
B_{k\geq3}:&\quad a_{0,3} = 0,\ \uwave{a_{2,1} \ne0},\ 3a_{0,5}\,a_{2,1} - 5a_{1,3}^2 = 0,\ \xi_3 = \cdots = \xi_{k-1} = 0,\ \xi_k \ne0,\\
C_k:&\quad a_{0,3} = 0,\ \uwave{a_{2,1} = \cdots = a_{k-1,1} = 0,\ a_{k,1} \ne0},\ a_{1,3} \ne0,\\
F_4:&\quad a_{0,3} = 0,\ \uwave{a_{2,1} = 0,\ a_{3,1} \ne0},\ a_{1,3} = 0,\ a_{0,5} \ne0,
\end{align*}
where
\[
\xi_n = \sum_{i=0}^n \sum_{j\geq 1} \dfrac{a_{i,2j-1}\,c_2^{m_2}\cdots c_n^{m_k}}{m_2!\cdots m_k!\,(2j-1)!},\quad \sum_{l=2}^n m_l = i,\quad \sum_{l=2}^n (l-1)m_l = n - j + 1
\]
and $c_2$, $\ldots$, $c_k$ are constants determined by 
\[
\sum_{i = 1}^n \sum_{j\geq 1} \dfrac{a_{i,2j-1}\,c_2^{l_2}\,c_3^{l_3}\cdots c_n^{l_n}}{l_2!\,l_3!\cdots l_n!\,(2n-1)!}=0,\quad \sum_{m=2}^n l_m= i -1,\quad \sum_{m=2}^n (m-1)l_m =n-j,\quad n=2,\ldots,k.
\]
\end{proposition}

\begin{remark}
These criteria are also shown in [16, page 707]. 
However, their criterion are not complete since they describe the criteria up to 4-jets.
We obtained Proposition \ref{prop:CriteriaOfA} without knowing their result and feel better 
to present our proof for completeness.
\end{remark}

\begin{proof}
[Proof of Proposition \ref{prop:CriteriaOfA}]
Note that the determinacy degrees of $S_k$, $B_k$, $C_k$ and $F_4$ are shown in \cite[Table 3]{Mond1985}.

For $S_1$\,:
We first remark that $S_1$-singularity is 3-$\mathcal{A}$-determined.
The left coordinate changes
\begin{align}
\label{eq:CoordinateChanges01}
\hat{y} = y - \sum_{i=2}^3 \frac{b_i}{i!} x^i,\quad \hat{z} = z - \sum_{i=2}^3 \frac{a_{i,0}}{i!} x^i - a_{1,2}\left(y - \sum_{i=2}^3 \frac{b_i}i x^i \right)
\end{align}
reduces $j^3g(0)$ to
\[
\left(u,\, \frac12 v^2,\, \frac12 a_{2,1} u^2 v + \frac16 a_{0,3} v^3\right),
\]
and this implies that $g$ is $\mathcal{A}$-equivalent to $S_1$ if and only if $a_{2,1} \ne0$ and $a_{0,3} \ne0$.\\
\indent
For $S_{\geq2}$\,: 
We first remark that $S_k$ is $(k+2)$-determined.
By the left coordinate changes like \eqref{eq:CoordinateChanges01}, we
reduce $j^{k+2}g(0)$ to
\begin{align}
\label{eq:kjet}
\left(u,\, \frac12 v^2,\, \sum_{i+2j=2}^{k+1} \frac{a_{i,2j+1}}{i!(2j+1)!} u^i v^{2j+1} \right),\quad (i,\ j\geq 0).
\end{align}
If $g$ is $\mathcal{A}$-equivalent to $S_k$, we may assume that $a_{0,3} \ne0$.
Since $a_{0,3} \ne0$, we can choose $c$ so that the coefficients of $u^i v^{2j+3m}$ $(i+j\geq 1 ,m\geq 1)$ of the third component of \eqref{eq:kjet} are zero by the left coordinate change $\hat{z} = z - c x^i y^j z^m$.
Hence we reduce \eqref{eq:kjet} to
\[
\left(u,\, \frac12 v^2,\, \sum_{i=2}^{k+1} \frac{a_{i,1}}{i!} u^i v + \frac16 a_{0,3} v^3\right),
\]
and we have $a_{2,1} = \cdots = a_{k,1} = 0$ and $a_{k+1,1} \ne0$.
On the other hand, if $a_{2,1} = \cdots = a_{k,1} = 0$, $a_{k+1,1} \ne0$ and $a_{0,3} \ne0$, then we can reduce \eqref{eq:kjet} to
\[
\left(u,\,
\frac12 v^2,\, \frac{a_{k+1,1}}{(k+1)!} u^{k+1} v + \frac16 a_{0,3} v^3\right)
\]
by the above left coordinate change $\hat{z} = z - c x^i y^j z^m$, and this implies that $g$ is $\mathcal{A}$-equivalent to $S_k$-singularity.\\
\indent
For $B_2$\,:
We first remark that $B_2$-singularity is 5-$\mathcal{A}$-determined.
By the left coordinate changes like \eqref{eq:CoordinateChanges01}, we reduce $j^5g(0)$ to
\begin{align}
\label{eq:5jetOfB2}
\text{\small{$
\left(u,\, \frac12 v^2,\, \frac12 a_{2,1} u^2 v + \frac16 a_{0,3} v^3 + \frac16 a_{3,1} u^3 v + \frac16a_{1,3}u v^3 + \frac1{24} a_{4,1} u^4 v + \frac1{1,2} a_{2,3} u^2 v^3 + \frac1{120} a_{0,5} v^5\right).$}}
\end{align}
If $g$ is $\mathcal{A}$-equivalent to $B_2$, we may assume that $a_{2,1}\ne0$.
Since $a_{2,1}\ne0$, substituting $u$ by $u-a_{1,3}v^2/(6a_{2,1})$,
we reduce the coefficient of $u v^3$ of the third component of \eqref{eq:5jetOfB2} to zero. 
Moreover, using the left coordinate change $\hat{z} = z - c x^i y^j z^m$ $(i+j\geq 1, m\geq 1)$ and choosing a suitable coefficient $c$, we can reduce the coefficients of $u^{i+2m}v^{2j+m}$ of the third component of \eqref{eq:5jetOfB2} to zero.
Hence we reduce \eqref{eq:5jetOfB2} to
\[
\left(u,\, \frac12 v^2,\, \frac12 a_{2,1} u^2 v + \frac16 a_{0,3}v^3 + \frac{3a_{2,1}\,a_{0,5} - 5a_{1,3}^2}{360a_{2,1}} v^5\right),
\]
and we have $a_{0,3} = 0$ and $3a_{2,1}\,a_{0,5} - 5a_{1,3}^2 \ne0$.
On the other hand, if $a_{2,1} \ne0$, $a_{0,3} = 0$ and $3a_{2,1}\,a_{0,5} - 5a_{1,3}^2 \ne0$, then by the above changes of coordinate of the source and the target we can reduce
\eqref{eq:5jetOfB2} to 
\[
\left(u,\,
\frac12 v^2,\, \frac12 a_{2,1} u^2 v + \frac{3a_{2,1}\,a_{0,5} - 5a_{1,3}^2}{360a_{2,1}} v^5\right), 
\]
and we conclude that $g$ is $\mathcal{A}$-equivalent to $B_2$. \\
\indent
For $B_{\geq 3}$\,:
First, we remark that $B_k$ singularity is $(2k+1)$-determined.
By the left coordinate changes
\[
\hat x = x,\quad \hat y = y - \sum_{i=2}^{2k+1}\dfrac{b_i}{i!}x^i,\quad \hat z = z - \sum_{i+2j=2}^{2k+1}\dfrac{2a_{i,2j}}{i!(2j)!}x^i \left(y - \sum_{i=2}^{2k+1}\dfrac{b_i}{i!}x^i\right)^j,
\]
the coefficients $u^i$ of the second component and $u^i v^{2j}$ of the third component of $j^{2k+1}f(0)$ became to 0, and thus $j^{2k+1}f(0)$ reduces to
\begin{equation}
\label{eq:2k+1jet}
\left(u,\, \frac12v^2,\, \dfrac{a_{2,1}}2u^2 v + \dfrac{a_{0,3}}6 v^3 + \sum_{i+2j=3}^{2k}\frac{a_{i,2j+1}}{i!(2j+1)!}u^iv^{2j+1}\right).
\end{equation}
Remark that this coordinate changes does not change all coefficients of the third component except coefficients $u^i v^{2j}$.
 
Assume that $f$ is $\mathcal{A}$-equivalent to $B_k$ singularity.
From \eqref{eq:2k+1jet}, we have $a_{2,1} \ne 0$ and $a_{0,3} = 0$.
Replacing $u$ by $u+\sum_{i=2}^k c_i v^{2(i-1)}$, we write $j^{2k+1}f(0)$ as 
\begin{align}
\label{eq:2k+1jet2}
\Biggl(&u + \sum_{i=2}^k c_i v^{2(i-1)},\, \dfrac12 v^2,\,  \dfrac12 a_{2,1} u^2 v + \sum_{i+j=4}^{2k+1}\hat a_{i,j}u^i v^j\Biggr). 
\end{align}
Since $a_{2,1}\ne0$, we can choose $c_2$, $\ldots$, $c_k$ so that the coefficients of $u v^3$, $\ldots$, $u v^{2k-1}$ of the third component of \eqref{eq:2k+1jet2} became to zero. 
In fact, since the coefficients $\hat a_{1,3}$, $\ldots$, $\hat a_{1,2k-1}$ of $u v^3$, $\ldots$, $u v^{2k-1}$ of the third component are give by
\begin{align}
\label{eq:a13} \hat a_{1,3} & = \dfrac{a_{1,3}}{3!} + \dfrac{a_{2,1}}{2!}\dfrac{2!}{1!1!}c_2,\\
\hat a_{1,5} & = \dfrac{a_{1,5}}{5!} + \dfrac{a_{2,1}}{2!}\dfrac{2!}{1!1!}c_4 + \dfrac{a_{2,3}}{2!3!}\dfrac{2!}{1!1!}c_2 + \dfrac{a_{3,1}}{3!}\dfrac{3!}{1!2!}c_2^2,\\
\nonumber
&\quad\vdots\\
\hat a_{1,2k-1} & = \sum_{i = 1}^k \sum_{j\geq 1}\dfrac{a_{i,2j-1}\,c_2^{l_2}\,c_3^{l_3}\cdots c_k^{l_k}}{l_2!\,l_3!\cdots l_k!\,(2k-1)!}\quad \left(\sum_{m=2}^k l_m = i-1,\quad \sum_{m=2}^k (m-1)l_m = k-j\right),
\end{align}
$c_2$, $\ldots$, $c_k$ are determined by $\hat a_{1,3} = \cdots = \hat a_{1,2k-1} = 0$.
In particular, $c_2 = -a_{1,3}/(6a_{2,1})$.
The coefficients of $\hat a_{0,5}$, $\ldots$, $\hat a_{0,2k+1}$ of $v^5$, $\ldots$, $v^{2k+1}$ of the third component are given by
\begin{align}
\hat a_{0,5} & = \dfrac{a_{0,5}}{5!} + \dfrac{a_{1,3}}{3!}\dfrac{1!}{1!}c_2^1 + \dfrac{a_{2,1}}{2!}\dfrac{2!}{2!}c_2^2,\\
\hat a_{0,7} & = \dfrac{a_{0,7}}{7!} + \dfrac{a_{1,3}}{3!}\dfrac{1!}{1!}c_3^1 + \dfrac{a_{1,5}}{5!}\dfrac{1!}{1!}c_2^1 + \dfrac{a_{2,1}}{2!}\dfrac{2!}{1!1!}c_2^1 c_3^1 + \dfrac{a_{2,3}}{2!3!}\dfrac{2!}{2!}c_2^2 + \dfrac{a_{3,1}}{3!}\dfrac{3!}{3!}c_2^3,\\\nonumber
&\vdots\\
\hat a_{0,2k+1} & = \sum_{i=0}^k \sum_{j\geq 1} \dfrac{a_{i,2j-1}\,c_2^{m_2}\cdots c_k^{m_k}}{m_2!\cdots m_k!\,(2j-1)!}\quad \left(\sum_{l=2}^k m_l = i,\quad \sum_{l=2}^k (l-1)m_l = k-j+1\right).
\end{align}
In particular, $\hat a_{0,5} = (3a_{0,5}\,a_{2,1} - 5a_{1,3}^2)/(360a_{2,1})$. 
By the left change of coordinates
\[
\hat{\hat x} = \hat x - \sum_{i=2}^k 2^{i-1} c_i \hat y^{i-1},\,\quad \hat{\hat y} = \hat y,\,\quad \hat{\hat z} = \hat z - \sum c_{i,j,m} \left(\hat x - \sum_{i=2}^k 2^{i-1} c_i \hat y^{i-1} \right)^i \hat y^j \hat z^m, 
\]
coefficients of $u^2$, $\ldots$, $u^{2(k-1)}$ of the first component became to zero.
Moreover, since $a_{2,1} \ne 0$, we can choose $c_{i,j,m}$ so that the coefficients of $u^{i+2m} v^{2j+m}$ $(i\geq 0,\,j\geq 0,\,m\geq1,\,4\leq i+2j+3m\leq 2k+1)$ of the third component became to zero, and thus the third component reduces to
\[
\dfrac12 a_{2,1}u^2 v + \hat a_{0,5} v^5 + \sum_{m=3}^k \sum_{i=2}^{m-1}\hat a_{0,2i-1}c_{1,m-i-1,1} u v^{2m-1} + \sum_{m=3}^k\left(\hat a_{0,2m+1} + \sum_{i=3}^m \hat a_{0,2i-1}c_{0,m-i+1,1} v^{2m+1}\right).
\]
Setting $\hat a_{0,2n+1} = \xi_n$, we obtain $\xi_2 = \xi_3 = \cdots = \xi_{k-1} = 0$, $\xi_k \ne 0$.

Conversely, if $a_{2,1} \ne0$, $a_{0,3} = \xi_2 = \cdots = \xi_k = 0$, and $\xi_k\ne0$, then by the above changes of coordinate of the source and the target we can reduce \eqref{eq:2k+1jet} to
\[
\left(u,\,\dfrac12 v^2,\, \dfrac12 a_{2,1} u^2 v + \xi_k v^{2k+1}\right).
\]
Hence, $f$ is $\mathcal{A}$-equivalent to $B_k$. 

\indent
For $C_{\geq 4}$\,:
We first remark that $C_k$-singularity is $(k+1)$-$\mathcal{A}$-determined.
If $g$ is $\mathcal{A}$-equivalent to $C_k$, we may assume that $a_{1,3} \ne0$.
Since $a_{1,3} \ne0$, substituting $u$ by $u - \frac{6a_{0,2i+3}}{(2i+3)!\,a_{1,3}} v^{2i}$ $(i\geq 1)$, we reduce the coefficient of $v^{2i+3}$ of the third component of $j^{k+1}g(0)$.
Moreover, we can choose $c$ so that the coefficients of $u^{i+m}v^{2j+3m}$ $(i+j\geq 1, m\geq 1)$ of the third component of  $j^{k+1}g(0)$ are zero by the left coordinate change  $\hat{z} = z - c x^i y^j z^m$. 
Hence we can reduce $j^{k+1}g(0)$ to
\[
\left(u,\, \frac12 v^2,\, \frac16 a_{0,3} v^3 + \frac16 a_{1,3}u v^3 + \sum_{i=2}^k \frac{a_{i,1}}{i!} u^i v\right).
\]
This shows that $a_{2,1} = \cdots = a_{k-1,1} = 0$, $a_{k,1} \ne0$ and $a_{0,3} = 0$.
Conversely, if $a_{2,1} = \cdots = a_{k-1,1} = 0$, $a_{k,1} \ne0$,
$a_{0,3} = 0$ and $a_{1,3} \ne0$, then by the above changes of coordinate of  the source and the target we can reduce $j^{k+1}g(0)$ to
\[
\left(u,\, \frac12 v^2,\, \frac16 a_{1,3} u v^3 + \frac{a_{k,1}}{k!} u^k v \right),
\]
and we conclude that $g$ is $\mathcal{A}$-equivalent to $C_k$.\\
\indent
For $F_4$\,:
We first remark that $F_4$-singularity is 5-$\mathcal{A}$-determined.
By the left coordinate changes like \eqref{eq:CoordinateChanges01}, we reduce $j^5g(0)$ to
\begin{align}
\label{eq:5jetOfF4}
\text{\small{$\left(u,\, \frac12 v^2,\, \frac12 a_{2,1} u^2 v + \frac16 a_{0,3} v^3 + \frac16 a_{3,1} u^3 v + \frac16 a_{1,3} u v^3 + \frac1{24} a_{4,1} u^4 v + \frac1{12} a_{2,3} u^2 v^3 + \frac1{120} a_{0,5} v^5\right).$}}
\end{align}
If $g$ is $\mathcal{A}$-equivalent to $F_4$, we may assume that $a_{3,1}\ne0$.
Since $a_{3,1}\ne0$, replacing $u$ to $u - a_{2,3}/(6a_{3,1}) v^2$, we see that the coefficient of $u^2 v^3$ of the third component of \eqref{eq:5jetOfF4} reduces to zero.
Moreover, by the left coordinate change $\hat{z} = z - a_{4,1}/(4a_{3,1}) x z$, we can reduce the coefficient of $u^4v$ of the third component of \eqref{eq:5jetOfF4} to zero.
Hence $j^5g(0)$ reduces to
\[
\left(u,\, \frac12 v^2,\, \frac16 a_{2,1} u^2 v + \frac16 a_{0,3} v^3 + \frac16 a_{3,1} u^3 v + \frac16 a_{1,3} u v^3 + \left(\frac1{120} a_{0,5} - \frac{a_{1,3}\,a_{2,3}}{36a_{3,1}}\right)v^5\right). 
\]
This implies that we have $a_{2,1} = a_{0,3} = a_{1,3} = 0$ and $a_{0,5} \ne0$.
Conversely, if $a_{2,1} = a_{0,3} = a_{1,3} = 0$, $a_{3,1} \ne0$ and $a_{0,5} \ne0$, then by the above changes of coordinate of
the source and the target we can reduce $j^5g(0)$ to
\[
\left(u,\, \frac12 v^2,\, \frac16 a_{3,1} u^3 v + \frac1{120} a_{0,5} v^5\right),
\]
and thus $g$ is $\mathcal{A}$-equivalent to $F_4$.
\end{proof}

\subsection{Basic notions of differntial geometry of singular surfaces with corank 1 singularities.}
  
Consider a singular surface $S$ parameterized by a smooth map-germ $g:(\R^2, 0)\to(\R^3, 0)$ of corank 1 at the origin $0$. 
At the singular point $g(0)$, the tangent plane degenerates to a line, that is, the image of $dg_0$ is a line.
We call such a line a {\it tangent line}. 
The plane passing through $g(0)$ perpendicular to the tangent line is called the {\it normal plane}.

\textR{We consider the orthogonal projection of $S$ onto the normal plane. 
The projection can be expressed as
\[
(\R^2,0)\to(\R^2,0),\quad(u,v)\mapsto(p(u,v),q(u,v)).
\]
We consider the group $\mathcal{G} = \mathrm{GL}_2(\R) \times \mathrm{GL}_2(\R)$ which acts on $(j^2 p, j^2 q)$. 
The list of $\mathcal{G}$-orbits is given in Table \ref{tab:singular_point} (see \cite{Gibson1979} for example). 
We classify the singular points of $S$ on the basis of the $\mathcal{G}$-class of $(j^2 p, j^2 q)$ in Table \ref{tab:singular_point}.
From Proposition \ref{prop:normal_form}, if $j^2 g(0)$ is $\mathcal{A}$-equivalent to $(u,v^2,0)$ then the singular point $g(0)$ is a hyperbolic, inflection or degenerate inflection point.
On the other hand, if $j^2 g(0)$ is $\mathcal{A}$-equivalent to $(u,u v,0)$, then the singular point $g(0)$ is either a parabolic or inflection point. 
\begin{table}[ht!]
\centering
\caption{The classification of the singular points.}
\begin{tabular}{cc}
\Bline
$\mathcal{G}$-class & Name \\\hline
$(x^2,y^2)$ & hyperbolic point\\
$(x y,x^2-y^2)$ & elliptic point\\
$(x^2,x y)$ & parabolic point\\
$(x^2\pm y^2,0)$ & inflection point\\
$(x^2,0)$ & degenerate inflection point\\
$(0,0)$ & degenerate inflection point\\\Bline
\end{tabular}
\label{tab:singular_point}
\end{table}}

There exists non-zero vector $\eta\in T_0\R^2$ such that $dg_0(\eta)=0$.
We call $\eta$ a {\it null vector} (cf. \cite{KRSUY2005,Saji2010}).
Suppose that $j^2{g}(0)$ is $\mathcal{A}$-equivalent to $(u,v^2,0)$.
The plane passing through $g(0)$ spanned by the tangent line and $\eta\eta g(0)$ is called the {\it principal plane}, where $\eta\eta g$ is the twice times directional derivative of $g$ with respect to $\eta$.
The vector in the normal plane is called the {\it principal normal vector} if the vector is normal to the principal plane. 

We remark that the definitions of the tangent line, normal plane, principal plane, principal normal vector and type of singular points are independent of the choice of coordinates in the source and choice of $\eta$. 

A regular plane curve in the parameter space passing through $(0,0)$ is called a {\it tangential curve} if it is transverse to $\eta$ at $(0,0)$. 
Let $\gamma(t)$ be a parameterization of the tangential curve. 
Clearly, $g\circ\gamma$ is tangent to the tangent line of the singular
surface. 
We denote $\Gamma$ by a family of tangential curves $\gamma$.
A member $\Gamma_0$ of the family is a {\it characteristic tangential curve} if the curvature of the orthogonal projection of $g\circ\Gamma_0$ onto the principal plane at $g(0)$ has an extremum value $\kappa_0$.
Note that tangential curves tangent to the characteristic tangential curve are characteristic tangential curves.

\begin{figure}[ht!]
\begin{minipage}[t]{0.49\textwidth}
\centering
\includegraphics[width=0.75\textwidth,clip]{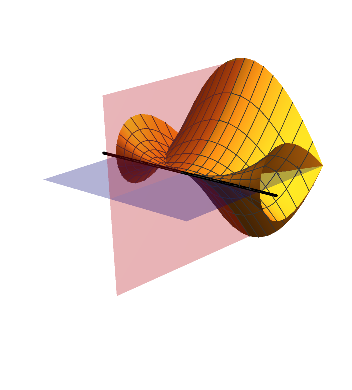}
\end{minipage}\hfill
\begin{minipage}[t]{0.49\textwidth}
\centering
\includegraphics[width=0.75\textwidth,clip]{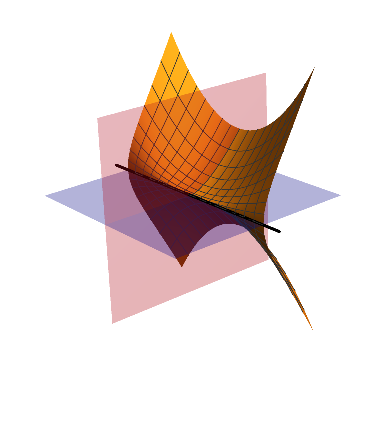}
\end{minipage}
\caption{The tangent line, normal plane and principal plane of $S^-$  (left) and $S^+$ (right).}
 \end{figure}

\begin{remark}
\label{rem:Normal}
Assume that a singular surface is parameterized by $g:(\R^2,0)\to(\R^3,0)$ given in \eqref{eq:normal_form_corank1}.  
We can easily show that the tangent line is the $x$-axis and the normal plane is the $yz$-plane, where $(x,y,z)$ is the usual
Cartesian  coordinate system of $\R^3$. 
Furthermore, the null vector can be chosen as $\eta = \partial_v$, and thus the principal plane is the $xy$-plane $\pm \partial_z$ are the principal normal vectors.

We can take  $\Gamma=(u, c_1 u + c_2 u^2 + O(u^3))$ as the family of tangential curves.
The 2-jet of $g\circ\Gamma$ are given by $(u, (b_2 + c_1^2) u^2/2, a_{2,0} u^2/2)$.
It follows that tangential curves tangent to the $u$-axis are the characteristic tangential curves, and thus the singular point $g(0)$ is an inflection (resp. degenerate inflection) point if and only if $a_{2,0} = 0$ (resp. $a_{2,0} = b_2 = 0$).
By using the above argument, it is easily shown that the singular point $g(0)$ is an inflection point if and only if $g\circ\gamma$ have at least 3-point contact (inflectional tangent) with the principal plane at $g(0)$, and that the inflection point $g(0)$ is degenerate if and only if $\kappa_0 = 0$.
\end{remark}

\subsection{Extended geometric properties of singular surfaces via blowing-ups}
  
Let $S$ be a singular surface parameterized by $g$ in \eqref{eq:normal_form_corank1}, and let $g$ be $\mathcal{A}$-equivalent to one of $S_k$, $B_k$, $C_k$ and $F_4$ singularities.
From Proposition \ref{prop:CriteriaOfA}, the condition that
\begin{align}
\label{eq:assumption}
a_{2,1} \ne0 \quad\text{or}\quad a_{2,1} = \cdots = a_{n,1} = 0,\ a_{n+1,1} \ne0  \quad\text{for some}\quad n\geq2. 
\end{align}
holds.
Consider maps
\[
\wt{\Pi}_{n+1}:\R\times S^1\to\R^2, \quad (r,\theta)\mapsto(r\cos\theta,r^{n+1}\cos^n\theta\sin\theta)\quad(n=1\text{ if }a_{21}\ne0),
\]
and
\[
\Pi_{n+1}:\mathcal{M}\to\R^2,\quad [(r,\theta)]\mapsto(r\cos\theta,r^{n+1}\cos^n\theta\sin\theta)\quad(n=1\text{ if }a_{21}\ne0), 
\]
where $\mathcal{M}=\R\times S^1/(r,\theta)\sim(-r,\theta+\pi)$.
The exceptional set $X=\Pi_{n+1}^{-1}(0,0)=\{(r,\theta)\,|\,r \cos\theta =0\}$.
Note that
\begin{equation}
\label{eq:pullback}
{\wt{\Pi}_{n+1}}^*u^i v^j=(u^i v^j)\circ\wt{\Pi}_{n+1}=r^{i+j+j n}\cos^{i+j n}\theta\sin^j\theta.
\end{equation}

Set
\[
A_m = A_m(u,v) = \sum_{i+j=m}\dfrac{a_{i,j}}{i!j!}u^i v^j,\quad
A = A(u,v) = \sum_{m = 3}^k A_m,\quad
B = B(u) = \sum_{i = 3}\dfrac{b_i}{i!}u^i.
\]
We have
\begin{equation}
\label{eq:FirstDerivatives}
g_u = (1, B_u, a_{2,0} u + A_u), \quad
g_v = (0, v, A_v),
\end{equation}
and we thus obtain
\begin{equation}
\label{eq:Cross}
g_u \times g_v = (A_v B_u - a_{2,0} u v - v A_v, -A_v, v).
\end{equation}
Therfore, we have
\begin{align}
\label{eq:blow_ext}
\begin{split}
& \text{\footnotesize{
$\wt{\Pi}^*_{n+1}(g_u\times g_v)$}} \\
& = \text{\footnotesize{$r^{n+1}\cos^n\theta\Biggl(\left(\frac{a_{n+1,1}\, b_2}{(n+1)!}\cos^2\theta  - a_{2,0}\cos\theta\sin\theta\right)r$}}\\
&\quad\quad\quad\quad\quad\quad\text{\footnotesize{$+\left(\dfrac{2 a_{n+2,1}\,b_2 + a_{n+1,1}\,b_3}{2(n+1)!}\cos^3\theta +\dfrac{2a_{1,2}\,b_2 - a_{3,0}}2\cos^2\theta\sin\theta\right) r^2 + O(r^3)$}},\\
&\quad\quad\quad\quad\quad\quad\text{\footnotesize{$ - \frac{a_{n+1,1}}{(n+1)!}\cos\theta + \left(-\frac{a_{n+2,1}}{(n+2)!}\cos^2\theta - a_{1,2}\cos\theta\sin\theta\right)r$
}}\\
&\quad\quad\quad\quad\quad\quad\quad\text{\footnotesize{$ + \left(\frac{a_{n+3,1}}{(n+3)!}\cos^2\theta + \frac12a_{2,2}\cos\theta\sin\theta - \varepsilon\frac12a_{0,3}\sin^2\theta\right)r^2 + O(r^3),\ \sin\theta\Biggr)$}},
\end{split}
\end{align}
where $\varepsilon=1$ if $n=1$, or $\varepsilon=0$ if $n\geq 2$.
Write the unit normal vector $\wt{\bm{n}} = \wt{\Pi}_{n+1}^*\bm{n}$ in the form
\[
\wt{\bm{n}}(r,\theta)=(n_1(r,\theta),\ n_2(r,\theta),\ n_3(r,\theta)).  
\]
Using \eqref{eq:blow_ext}, we show that $\wt{\bm{n}}$ can be extendible near the exceptional set $X$ and $n_i$ can be written in
\begin{align*}
\begin{split}
n_1 & = O(r),\\
n_2 & = n_{20} + n_{21}r + n_{22}r^2 + O(r^3),\\
n_3 & = n_{30} + n_{31}r + n_{32}r^2 + O(r^3),
\end{split}
\end{align*}
where
\[
n_{20} = -\dfrac{a_{n+1,1}\cos\theta}{\ma(\theta)},\quad
n_{30} = \dfrac{(n+1)!\sin\theta}{\ma(\theta)}, 
\]
and the coefficients $(n_{11}, n_{21}, n_{22}, n_{31}, n_{32})$ are trigonometric polynomials with coefficients depending on the $4$-jet and $a_{i,1}$ $(n+1\leq i \leq n+3)$ of $g$, expressed in \eqref{eq:blow_n21} to \eqref{eq:blow_n32} in Appendix A. 
Here, 
$$
\ma(\theta) = \sqrt{a_{n+1,1}^2\cos^\theta+((n+1)!)^2\sin^2\theta}.
$$
 
\begin{remark}
When the singular surface $S$ is parameterized by a map-germ $\mathcal{A}$-equivalent to $H_k$, we cannot obtain an extended unit normal vector of $S$ via such a map $\wt{\Pi}_{n+1}$, and the expressions of the second fundamental form below do not work. 
This is the reason why we avoid the case in this paper.  
\end{remark}

Assume that $\wt{\bm{n}}(0,\theta)$ is not the principal normal vector, that is, $\cos\theta \ne 0$. 
Let us obtain the pull backs of the coefficients $E$, $F$ and $G$ of the first fundamental form of $S$. 
From \eqref{eq:pullback} and \eqref{eq:FirstDerivatives} we have
\begin{align*}
\wt{\Pi}_{n+1}^* g_u & =\left(
1,\,b_2 r + \dfrac{b_3}2 r^2 +O(r^3),\, a_{2,0}r\cos\theta + a_{3,0}r^2\cos^2\theta + O(r^3)\right),\\
\wt{\Pi}_{n+1}^* g_v  & =r^{n+1}\left( 0,\,\cos^n\theta\sin\theta,\, \dfrac{a_{n+1,1}}{(n+1)!}\cos^{n+1}\theta + \left(\dfrac{a_{n+2,1}}{(n+2)!}\cos\theta + a_{1,2}\right)r\cos^{n+1}\theta\right),
\end{align*}
and thus
\begin{align}
\label{eq:pullback_of_E}
\wt{E} & = \wt{\Pi}^*_{n+1}E =1 + E_2 r^2 + O(r^3),\\
\wt{F} & = \wt{\Pi}^*_{n+1}F = r^{n+2}(F_0 + F_1 r + O(r^2)),\\
\label{eq:pullback_of_G}
\wt{G} & = \wt{\Pi}^*_{n+1}G = r^{2n+2}(G_0 + G_1 r + O(r^2)),
\end{align}
where
\begin{align*}
E_2 & =(a_{2,0}^2 + b_2^2)\cos^2\theta,\\
F_0 & = \left(\frac{a_{n+1,1}\, a_{2,0}}{(n+1)!}\cos\theta+ b_2\sin\theta\right)\cos^{n+1}\theta,\\
G_0 & = \left(\left(\frac{a_{n+1,1}}{(n+1)!}\right)^2\cos^2\theta + \sin^2\theta\right)\cos^{2n}\theta,\\
F_1 & = \left(\left(\frac{a_{n+2,1}\, a_{2,0}}{(n+2)!} +\frac{a_{n+1,1}\, a_{3,0}}{2(n+1)!}\right)\cos\theta + \left(a_{1,2}\,a_{2,0}+\frac12b_3\right)\sin\theta\right)\cos^{n+2}\theta,\\
G_1 & = \frac{2a_{n+1,1}}{(n+1)!}\left(\frac{a_{n+2,1}}{(n+2)!}\cos\theta +a_{1,2}\sin\theta\right)\cos^{2n+2}\theta.
\end{align*}

We now obtain the pull backs of the coefficients $L$, $M$ and $N$ of the second fundamental form of $S$.
We have
\begin{equation}
\label{eq:SecondDerivatives}
g_{u u} = (0, B_{u u}, a_{2,0} + A_{u u}),\quad
g_{u v} = (0, 0, A_{u v}),\quad
g_{v v} = (0, 1, A_{v v}).
\end{equation}
Using \eqref{eq:pullback} and \eqref{eq:SecondDerivatives}, we have
\begin{align*}
\wt{\Pi}_{n+1}^* g_{uu}& =\mbox{\footnotesize{$\left(0,\, \dsum{2}{4}\dfrac{b_i}{(i-2)!}r^{i-2}\cos^{i-2}\theta + O(r^3),\, \dsum{2}{4}\dfrac{a_{i,0}}{(i-2)!}r^{i-2}\cos^{i-2}\theta + a_{2,1} r^2 \cos\theta\sin\theta + O(r^3)\right)$}},\\
\wt{\Pi}_{n+1}^* g_{uv}& =\mbox{\small{$\left(0,\,0,\,\dsum{0}{2}\dfrac{a_{n+i+1,1}}{(n+i)!}r^{n+i}\cos^{n+i}\theta + \dsum{1}{2} a_{i,2} r^{n+i} \cos^{n+i-1}\theta\sin\theta + O(r^{n+3})\right)$}},\\
\wt{\Pi}_{n+1}^* g_{vv} & =\mbox{\small{$\left(0,\,1,\, \dsum{1}{2} \dfrac{a_{i,2}}{i!} r^i \cos^i\theta + \epsilon a_{0,3}r^2 \cos\theta \sin\theta + O(r^3)\right)$}},
\end{align*}
and thus
\begin{align}
\label{eq:pullback_of_L}
\wt{L} & = \wt{\Pi}^*_{n+1}L = L_0+L_1r+L_2r^2+O(r^3),\\
\wt{M} & = \wt{\Pi}^*_{n+1}M = r^n(M_0+M_1r+M_2r^2+O(r^3)),\\
\label{eq:pullback_of_N}
\wt{N} & = \wt{\Pi}^*_{n+1}N = N_0+N_1r+N_2r^2+O(r^3),
\end{align}
where
\begin{align*}
L_0 & = \frac{-a_{n+1,1}\, b_2\cos\theta +(n+1)!\,a_{2,0}\sin\theta}{\ma(\theta)},\\
M_0 & = \frac{(n+1)a_{n+1,1}\cos^n\theta\sin\theta}{\ma(\theta)},\\
N_0 & = -\frac{a_{n+1,1}\cos\theta}{\ma(\theta)},
\end{align*}
and the coefficients $(L_1, M_1, N_1, L_2, M_2, N_2)$ are trigonometric polynomials with coefficients depending on the 4-jet of $g$ and  $a_{i,1}$ $(n+1\leq i\leq n+3)$ expressed in \eqref{eq:blow_L1} to \eqref{eq:blow_N2} in Appendix A.

Since the Gaussian curvature $K$ is given by $K=(L N - M^2)/(E G - F^2)$, by using \eqref{eq:pullback_of_E} -- \eqref{eq:pullback_of_G} the Gaussian curvature $\wt{K} = \wt{\Pi}_{n+1}^*K$ in $(r,\theta)$ can be expressed as
\begin{align}
\label{eq:gaussian}
\wt{K} = \dfrac1{r^{2n+2}}(K_0 + K_1 r + K_2 r^2 + O(r^3)),
\end{align}
where
\begin{align*}
K_0 & = \dfrac{L_0 N_0}{G_0} =  \dfrac{((n+1)!)^2\,a_{n+1,1}(a_{n+1,1}b_2\cos\theta-(n+1)!\,a_{2,0}\sin\theta)}{\ma(\theta)^4\cos^{2n-1}\theta},\\
K_1 & = \dfrac{G_0 L_1 N_0 + G_0 N_0 N_1 - G_1 L_0 N_0 }{G_0^2},\\
K_2 & = \dfrac1{G_0^3}\left(F_0^2 G_0 L_2 N_0 + G_1^2 L_0 N_0 - E_2^2 G_0^2 L_0 N_0 - G_0 G_1 L_1 N_0\right.\\
&\left.\qquad\qquad + G_0^2 L_2 N_0 - G_0 G_1 L_0 N_1 + G_0^2 L_1 N_1 + G_0^2 L_0 N_2 - \epsilon G_0^2 M_0^2
\right).
\end{align*}
We say that a point $(0,\theta_0)$ is an \textit{elliptic}, \textit{hyperbolic} or \textit{parabolic point over the singularity} of $S$ if $r^{2n+2}\wt{K}(0,\theta_0) = K_0$ is positive, negative, or zero, respectively.

The principal curvatures $\kappa_1$ and $\kappa_2$ of $g$ are the roots of the equation
\[
(E G-F^2)\kappa+(-E N+2F M-G L)\kappa+(L N-M^2)=0.
\]
So $\kappa_i$ is given by
\[
\kappa_i = \dfrac{E L - 2F M + G L + \epsilon'\sqrt{(E L - 2F M + G L)^2-4(E G - F^2)(L N - M^2)}}{2(E G - F^2)},
\]
where $\epsilon' = 1$ if $i = 2$ or $\epsilon' = -1$ if $i = 1$.
We assume that $N_0>0$ (if $N_0<0$, we should change $\kappa_1$ with $\kappa_2$).
It follows from \eqref{eq:pullback_of_E} -- \eqref{eq:pullback_of_G} and \eqref{eq:pullback_of_L} -- \eqref{eq:pullback_of_N} that
\begin{align*}
\tk{1} & = \wt{\Pi}_{n+1}^*\kappa_1 = \dfrac1{r^{2n+2}}\left(\dfrac{N_0 - \sqrt{N_0^2}}{2G_0} + O(r)\right),\\
\tk{2} & = \wt{\Pi}_{n+1}^*\kappa_2 = \dfrac1{r^{2n+2}}\left(\dfrac{N_0 + \sqrt{N_0^2}}{2G_0} + O(r)\right).
\end{align*}
Hence, $\tk{2}$ can be expressed as
\begin{equation}
\label{eq:kappa2}
\tk{2} = \dfrac1{r^{2n+2}}(k_{20} + k_{21} r + k_{22} r^2 + O(r^3)),
\end{equation}
where
\[
k_{20} = \dfrac{N_0}{G_0} = -\frac{((n+1)!)^2a_{n+1, 1}}{\ma(\theta)^3 \cos^{2n - 1}\theta}.
\]
Since \eqref{eq:gaussian}, \eqref{eq:kappa2}, and $\wt{K} =
\tk{1}\tk{2}$, $\tk{1}$ can be expressed as
\begin{equation}
\label{eq:k1}
\tk{1} = k_{10} + k_{11} r + k_{12} r^2 + O(r^3).
\end{equation}
So we have
\begin{align*}
K_0 = k_{10} k_{20},\quad K_1 = k_{10} k_{21} + k_{11} k_{20},\quad K_2 = k_{10} k_{22} + k_{11} k_{21} + k_{12} k_{20}.
\end{align*}
These give
\begin{align}
k_{10} & = \dfrac{K_0}{k_{20}} = L_0 = \frac{-a_{n+1,1}\, b_2\cos\theta +(n+1)!\,a_{2,0}\sin\theta}{\ma(\theta)},\\
k_{11} & = \dfrac{K_1 - k_{10}k_{21}}{k_{20}} = L_1,\\
k_{12} & = \dfrac{K_2 - k_{10}k_{22} - k_{11}k_{21}}{k_{20}} = -E_2 L_0 + L_2 - \epsilon \dfrac{M_0^2}{N_0}.
\end{align}
The expressions of $k_{11}$ and $k_{12}$ expressed in the original coefficients in \eqref{eq:normal_form_corank1} are given respectively by \eqref{eq:blow_k11} and \eqref{eq:blow_k12} in Appendix.
 
The principal direction $(du,dv)$ corresponding to the principal curvature $\kappa_i$ is given by the equation
\[
\begin{pmatrix}
L & M \\
M & N
\end{pmatrix}
\begin{pmatrix}
du \\ dv
\end{pmatrix}
=
\kappa_i
\begin{pmatrix}
E & F \\
F & G
\end{pmatrix}
\begin{pmatrix}
du \\ dv
\end{pmatrix}.
\]
Hence, vectors $\bm{v}_i=(N-\kappa_iG)\partial_u-(M-\kappa_iF)\partial_v$ $(i=1,2)$ generate the principal directions corresponding to $\kappa_i$.
Since
\[
\wt{\Pi}_{n+1}(r,\theta) = (r \cos\theta, r^{n+1}\cos^n\theta \sin\theta) = (u,v),
\]
we have
\begin{align}
\label{eq:operator1}
\partial_u & = (\cos\theta-n\sin\theta\tan\theta)\partial_r - \frac1r(n+1)\sin\theta\partial_\theta,\\
\label{eq:operator2}
\partial_v & = \dfrac{\sin\theta}{r^n \cos\theta} \partial_r + \dfrac1{r^{n+1}}\cos^{1-n}\theta \partial_\theta.
\end{align}
Therefore, the lifted vectors $\tv{1}$ and $\tv{2}$, respectively, of $\bm{v}_1$ and $\bm{v}_2$ by $\wt{\Pi}_{n+1}$ are expressed as follows:
\begin{align}
\label{eq:v1}
\tv{1} & = (\xi_{10} + \xi_{11} r + O(r^2))\partial_r + (\eta_{10} + \eta_{11} r + O(r^2))\partial_\theta, \\
\label{eq:v2}
\tv{2} & = \dfrac1{r^{2n+1}}\left((\xi_{21} r + O(r^2)) \partial_r + (\eta_{20} + \eta_{21} r) \partial_\theta\right),
\end{align}
where
\begin{align*}
\xi_{10} &= N_0 (\cos\theta - n \sin\theta \tan\theta) - \dfrac{M_0\sin\theta}{\cos^n\theta}= -\frac{a_{n+1,1}}{\ma(\theta)}, \\
\xi_{11} & = N_1 (\cos\theta - n \sin\theta \tan\theta) - \dfrac{M_1 \sin\theta}{\cos^n\theta}, \\
\eta_{10} & = -(n+1)N_1\sin\theta - M_1\cos^{1-n}\theta = -\dfrac{((n+2)!\,a_{1,2}\sin\theta + a_{n+2,1}\cos\theta)\cos\theta\sin\theta}{(n+2)\ma(\theta)},\\
\eta_{11} & = -(n+1)N_2\sin\theta - (M_2 - k_{10} F_0) \cos^{1-n}\theta,\\
\xi_{21} & = - F_0 k_{20} \sin\theta = -\dfrac{(n+1)!\,a_{n+1,1} (a_{2,0}a_{n+1,1}\cos\theta + (n+1)!\,b_2\sin\theta)\sin\theta}{\ma(\theta)^3 \cos^{n-2}\theta},\\
\eta_{20} & = -F_0 k_{10} \cos\theta = -\dfrac{(n+1)!\,a_{n+1,1} (a_{2,0}a_{n+1,1}\cos\theta + (n+1)!\,b_2\sin\theta)}{\ma(\theta)^3 \cos^{n-1}\theta}.
\end{align*}
The expressions of $\xi_{11}$ and $\eta_{21}$ expressed in the original coefficients in \eqref{eq:normal_form_corank1} are given, respectively, by \eqref{eq:alpha} and \eqref{eq:beta} in Appendix. 

Since we have \eqref{eq:k1} and \eqref{eq:v1}, the first and second directional derivative of $\tk{1}$ along $\tv{1}$ can be expressed respectively as
\begin{align*}
\tv{1}\tk{1}(r,\theta)& = \xi_{10} k_{11} + \eta_{10} k_{10}' + 2\xi_{10}k_{12} + \xi_{11}k_{11} + \eta_{10}k_{11}' + \eta_{11}k_{10}' + O(r),\\
\tv{1}^2\tk{1}(r,\theta) &= 2\xi_{10}^2 k_{12} + \xi_{10} \xi_{11} k_{11} + \xi_{10} \xi_{11}k_{10}'\\
&\quad + 2 \xi_{10}\eta_{10}k_{11}' + \xi_{10}'\eta_{10}k_{11} + \eta_{10}\eta_{10}'k_{10} \eta_{11} k_{10}'' + O(r),
\end{align*}
where $'$ denotes the derivative with respect to $\theta$.
Moreover, the directional derivative of $\tk{1}$ along $\tv{2}$ can be expressed as
\[
\tv{2}\tk{1}(r,\theta)= \dfrac1{r^{2n+1}}(\eta_{20}k_{10}'+ O(r)).
\]
Therefore, we have
\begin{align*}
\tv{1}\tk{1}(r,\theta) & =
 \frac{a_{n+1,1}\Delta_1^{(n+1)}(\theta)\cos\theta}{\ma(\theta)^2}+O(r),\\
\tv{1}^2\tk{1}(r,\theta) & = \dfrac{a_{n+1,1} \bigl(a_{n+1,1}\Delta_2^{(n + 1)}(\theta)\cos\theta - (n + 1)!\, a_{1,2} \Delta_1^{(n+1)}(\theta)\sin\theta\bigr)\cos\theta}{\ma(\theta)^3} + O(r),\\ 
\tv{2}\tk{1}(r,\theta) & =\dfrac1{r^{2n+1}}\left(-\dfrac{\cos^{-2n+3}\theta((n+1)!)^2\,a_{n+1,1}^2\Delta_3^{(n+1)}(\theta)^2}{\ma(\theta)^3} + O(r)\right)
\end{align*}
where
\begin{align*}
\Delta_1^{(n+1)}(\theta) & = a_{n+1,1}\,b_3\cos\theta - (n+1)!\,a_{3,0}\sin\theta, \\
\Delta_2^{(n+1)}(\theta) & = - \bigl(a_{n+1,1}\,b_4 \cos\theta - (n+1)!\,a_{4,0} \sin\theta\bigr)\cos\theta\\
&\quad + 3(a_{2,0}^2 + b_2^2)\bigl(a_{n+1,1}\,b_2\cos\theta - (n+1)!\,a_{2,0}\sin\theta\bigr)\cos\theta + 12 a_{2,1}\sin^2\theta,\\
\Delta_3^{(n+1)}(\theta) & = a_{2,0}\,a_{n+1,1}\cos\theta_0 + (n+1)!\,b_2\sin\theta_0.
\end{align*}

A ridge point of a surface in $\R^3$ was first studied in details by Porteous \cite{Porteous1971} as a point where the distance squared function on the surface has an $A_{\geq 3}$-singularity.
It is also a point where one principal curvature has an extremum value along the corresponding line of curvature.
A point where one principal curvature has an extremum value along the other line of curvature is also important.
Such a point is called the sub-parabolic point, which was first studied in details by Bruce and Wilkinson \cite{BW1991} from the viewpoint of folding maps.
If a regular surface has a ridge point with respect to the line of curvature tangent to $\bm{v}_i$, then its focal surface corresponding to $\kappa_i$ has a singular point.
On the other hand, if a regular surface has a sub-parabolic point with respect to the line of curvature tangent to $\bm{v}_i$, then its focal surface corresponding to $\kappa_i$ has a parabolic point. 

We define the ridge and sub-parabolic points over the singularity of $S$ are as follows:
\begin{definition}
\label{prop:ridge}
\begin{enumerate}
\item
A point $(0,\theta_0)$ is a \textit{ridge point relative to $\tv{1}$ over the singularity of $S$} if $\Delta_1^{(n+1)}(\theta_0)=0$. 
Moreover, the \textit{ridge point $(0,\theta_0)$ is a first (resp. second or higher) order ridge point relative to $\tv{1}$ over the singularity of $S$} if $\Delta_2^{(n+1)}(\theta_0)\ne0$ (resp. $\Delta_2^{(n+1)}=0$). 

\item 
A point $(0,\theta_0)$ is a \textit{sub-parabolic point relative to $\tv{2}$ over the singularity of $S$} if $\Delta_3^{(n+1)}(\theta_0)=0$.
\end{enumerate}
\end{definition}

When $g$ is $\mathcal{A}$-equivalent to $S_1$, since $a_{2,1}\ne0$ (Proposition \ref{prop:CriteriaOfA}) we obtain $\wt{\bm{n}}$, $\tk{i\,}$ and $\tv{i\,}$ via $\wt{\Pi}_2$.
Similarly, when $g$ is $\mathcal{A}$-equivalent to one of $S_{\geq2}$, $B_k$, $C_k$ and $F_4$ singularities, we obtain $\wt{\bm{n}}$, $\tk{i\,}$, and $\tv{i\,}$ via $\wt{\Pi}_m$ as shown in Table \ref{tab:BlowUp}. 
Hence, we have the following lemma. 
\begin{table}[ht!]
\caption{Correspondence between the type of $\ma$-singularity and $\wt{\Pi}_n$.}
\centering
\begin{tabular}{ccccc}
\Bline
$\ma$-type & $S_k$ & $B_k$ & $C_k$ & $F_4$ \\\hline
$\wt{\Pi}_m$ & $\wt{\Pi}_{k+1}$ & $\wt{\Pi}_2$ & $\wt{\Pi}_k$ & $\wt{\Pi}_3$\\ 
\Bline
\end{tabular}
\label{tab:BlowUp}
\end{table}

\begin{lemma}
The necessary and sufficient conditions for a point $(0,\theta)$ to be a first ridge point relative to $\tv{1}$ and a sub-parabolic point relative to $\tv{2}$ over the singularity of $S$ are shown in Table \ref{tab:Ridge}.
\begin{table}[ht!]
\caption{Conditions for ridge and sub-parabolic points.}
\centering
\begin{tabular}{ccc}
\Bline
$\ma$-type & 1st ridge & sub-parabolic \\\hline
$S_k$ & $\Delta_1^{(k+1)}(0,\theta)=0$, $\Delta_2^{(k+1)}(0,\theta)\ne0$ & $\Delta_3^{(k+1)}(0,\theta)=0$\\
$B_k$ & $\Delta_1^{(2)}(0,\theta)=0$, $\Delta_2^{(2)}(0,\theta)\ne0$ & $\Delta_3^{(2)}(0,\theta)=0$\\
$C_k$ & $\Delta_1^{(k)}(0,\theta)=0$, $\Delta_2^{(k)}(0,\theta)\ne0$ & $\Delta_3^{(k)}(0,\theta)=0$\\
$F_4$ & $\Delta_1^{(3)}(0,\theta)=0$, $\Delta_2^{(3)}(0,\theta)\ne0$ & $\Delta_3^{(3)}(0,\theta)=0$\\
\Bline
\end{tabular}
\label{tab:Ridge}
\end{table}
\end{lemma}
  
\section{Families of distance squared functions on singular surfaces.}

We do not recall here the definition of a versal unfolding.
Refer, for example, to \cite[Section 8 and 19]{Arnold1986} and \cite[Section 3]{Wall1981}.

We define a family of functions $D$ on a surface $S$ parameterized by a smooth map-germ $g:(\R^2,0)\to (\R^3,0)$ by 
\[
 D:(\R^2,0) \times (\R^3,\vp_0) \to\R,\quad
 D(u,v,x,y,z)
 = \dfrac12\|g(u,v) - \vp\|^2.
\]
The function $d_{\vp_0}(u,v) = D(u,v,\vp_0)$ is the \textit{distance squared function} on $S$ from a point $\vp_0 = (x_0,y_0,z_0)$.

\begin{theorem}
\label{thm:Distance}
Let $g:(\R^2,0)\to(\R^3,0)$ be given in the form \eqref{eq:normal_form_corank1}, and let $g$ be $\ma$-equivalent to one of $S_k$, $B_k$, $C_k$ and $F_4$ singularities. 
Suppose that $\vp_0 = (x_0,y_0,z_0)$ is on the normal plane, that is, $\vp_0 \in \R \wt{\bm{n}}(0,\theta_0)$, where $\wt{\bm{n}}$ is the well-defined unit normal vector obtained by using $\wt{\Pi}_m$ given by Table \ref{tab:BlowUp} and $\theta_0\in(-\pi/2,\pi/2]$.
\begin{enumerate}
\item 
Suppose that $\vp_0$ is not on the principal normal line, and that $(0,\theta_0)$ is not a parabolic point over the singularity. 
\begin{enumerate}
\item[\upshape{(1a)}]
$d_{\vp_0}$ has an $A_1$-singularity at $(0,0)$ if and only if $\vp_0$ is not on the focal locus.  
When $d_{\vp_0}$ has an $A_1$-singularity at $(0,0)$, $D$ is an $\mathcal{R}^+$ and $\mathcal{K}$-versal unfolding of $d_{\vp_0}$. 
\item[\upshape{(1b)}]
$d_{\vp_0}$ has an $A_2$-singularity at $(0,0)$ if and only if $p_0$ is on the focal locus and $(0,\theta_0)$ is not a ridge point relative to $\tv{1}$ over the singularity.
When $d_{\vp_0}$ has an $A_2$-singularity at $(0,0)$, $D$ is an $\mathcal{R}^+$ and $\mathcal{K}$-versal unfolding of $d_{p_0}$. 
\item[\upshape{(1c)}]
$d_{\vp_0}$ has an $A_3$-singularity at $(0,0)$ if and only if $\vp_0$ is on the focal locus and $(0,\theta_0)$ is a first order ridge point relative to $\tv{1}$ over the singularity. 
When $d_{\vp_0}$ has an $A_3$-singularity at $(0,0)$, $D$ is an $\mathcal{R}^+$-versal unfolding of $d_{\vp_0}$, and $D$ is an $\mathcal{K}$-versal unfolding of $d_{\vp_0}$ if and only if $(0,\theta_0)$ is not a sub-parabolic point relative to $\tv{2}$ over the singularity. 
\item[\upshape{(1d)}]
$d_{\vp_0}$ has an $A_{\geq 4}$-singularity at $(0,0)$ if and only if $\vp_0$ is on the focal locus and $(0,\theta_0)$ is a second or higher order ridge point relative to $\tv{1}$ over the singularity.
When $d_{\vp_0}$ has an $A_{\geq5}$ or $A_{\geq4}$-singularity at $(0,0)$, $D$ is not an $\mathcal{R}^+$- or $\mathcal{K}$-versal unfolding of $d_{\vp_0}$, respectively. 
If $d_{\vp_0}$ has an $A_4$-singularity at $(0,0)$, then $D$ is an $\mathcal{R}^+$-versal unfolding of $d_{p_0}$ if and only if there exists $(0,\theta)\in(-\pi/2,\pi/2]$ such that $(0,\theta)$ is not a ridge point relative to $\tv{1}$. 
\end{enumerate}
\item 
Suppose that $\vp_0$ is on the principal normal line. 
Then $D$ is neither an $\mathcal{R}^+$- nor $\mathcal{K}$-versal unfolding of $d_{\vp_0}$.
\begin{enumerate}
\item[\upshape{(2a)}]
$d_{\vp_0}$ has an $A_{\geq 2}$-singularity at $(0,0)$ if and only if $\vp_0$ is not the intersection point of the focal locus.
\item[\upshape{(2b)}]
$d_{\vp_0}$ has a singularity of type $D_4$ or more degenerate at $(0,0)$ if and only if $\vp_0$ is the intersection point of the focal locus.
\end{enumerate}
\end{enumerate}
\end{theorem}

To show Theorem \ref{thm:Distance}, we first show criterion for singularities of distance-squared functions in terms of the coefficients in \eqref{eq:normal_form_corank1}.

\begin{proposition}
\label{prop:distance}
Let $g:(\R^2,0)\to(\R^3,0)$ be given in the form\eqref{eq:normal_form_corank1}. 
Then $d_{\vp_0}$ on $g$ is singular at $(0,0)$ if and only if $x_0=0$, that is, $\vp_0$ is on the normal plane.
Moreover, assume that $d_{\vp_0}$ is singular at $(0,0)$.
Then
\begin{enumerate}
\item 
$d_{\vp_0}$ has an $A_1$-singularity at $(0,0)$ if and only if $y_0(b_2y_0+a_{2,0}z_0-1)\ne0$.
When this is the case, $D$ is an $\mathcal{R}^+$- and $\mathcal{K}$-versal unfolding of $d_{\vp_0}$.
\item 
$d_{\vp_0}$ has an $A_2$-singularity at $(0,0)$ if and only if one of the following conditions holds:
\begin{align*}
(\mathrm{2a})& \quad y_0\ne0, \quad b_2 y_0 + a_{2,0} z_0 - 1 = 0,\quad b_3 y_0 + a_{3,0} z_0 \ne0;\\
(\mathrm{2b})& \quad y_0 = 0,\quad a_{2,0} z_0 \ne 1,\quad a_{0,3}z_0\ne0.
\end{align*}
If condition \textup{(2a)} holds, then $D$ is an $\mathcal{R}^+$- and $\mathcal{K}$-versal unfolding of $d_{\vp_0}$.
On the other hand, if condition \textup{(2b)} holds, then $D$ is neither an $\mathcal{R}^+$- nor $\mathcal{K}$-versal unfolding of $d_{\vp_0}$.
\item 
$d_{p_0}$ has an $A_3$-singularity at $(0,0)$ if and only if one of the following conditions holds:
\begin{align*}
(\mathrm{3a})& \quad y_0 \ne0,\quad b_2 y_0 + a_{2,0} z_0 - 1 = b_3 y_0 + a_{3,0} z_0 = 0,\\
&\quad b_4 y_0^2 + a_{4,0} y_0 z_0 - 3 a_{2,1}^2 z_0^2 - 3(a_{2,0}^2 + b_2^2) y_0 \ne0;\\
(\mathrm{3b})&\quad y_0 = 0,\quad a_{2,0} z_0 \ne 1,\quad a_{0,3} z_0 = 0,\quad (a_{0,4}\,a_{2,0} - 3a_{1,2}^2)z_0^2 - (a_{0,4} + 3 a_{2,0}) z_0 + 3 \ne0.
\end{align*}
If condition \textup{(3a)} holds, then $D$ is an $\mathcal{R}^+$-versal unfolding of $d_{\vp_0}$, and $\mathcal{K}$-versal unfolding of $d_{\vp_0}$ if and only if $a_{2,0} y_0 - b_2 z_0 \ne0$. 
On the other hand, if condition \textup{(3b)} holds, then $D$ is neither an $\mathcal{R}^+$- nor $\mathcal{K}$-versal unfolding of  $d_{\vp_0}$. 
\item 
$d_{\vp_0}$ has an $A_{\geq 4}$-singularity at $(0,0)$ if and only if one of the following conditions holds:
\begin{align*}
(\mathrm{4a})& \quad y_0 \ne0,\ b_2 y_0 + a_{2,0} z_0 - 1 = 0,\
b_3 y_0 + a_{3,0} z_0 = 0,\\
&\quad b_4 y_0^2 + a_{4,0} y_0 z_0 - 3 a_{2,1}^2 z_0^2 - 3(a_{2,0}^2 + b_2^2) y_0 = 0;\\
(\mathrm{4b})& \quad y_0 = 0,\ a_{2,0} z_0 \ne 1,\ z_0\ne 0,\ a_{0,3} = 0,\ (a_{0,4}\,a_{2,0} - 3 a_{1,2}^2)z_0^2 - (a_{0,4} + 3 a_{2,0}) z_0 + 3 =0.
\end{align*}
When $d_{\vp_0}$ has an $A_{\geq5}$- or $A_{\geq4}$-singularity at $(0,0)$, $D$ is not an $\mathcal{R}$- or $\mathcal{K}$-versal unfolding of $d_{\vp_0}$, respectively. 
If condition \textup{(4a)} holds, then $D$ is an
$\mathcal{R}^+$-versal unfolding of $d_{\vp_0}$ having an $A_4$-singularity at $(0,0)$ if and only if $(a_{3,0},b_3)\ne(0,0)$. 
On the other hand, if condition \textup{(4b)} holds, then $D$ is not an $\mathcal{R}^+$-versal unfolding of $d_{\vp_0}$. 
\item 
$d_{\vp_0}$ has a $D_4$ or more degenerate singularity at $(0,0)$ if and only if $y_0=0$ and $a_{2,0} z_0 = 1$.
When this is the case, $D$ is neither $\mathcal{R}^+$- nor $\mathcal{K}$-versal unfolding of $d_{\vp_0}$.
\end{enumerate} 
\end{proposition}

\begin{proof}
Since $\partial d_{\vp_0}/\partial u(0,0) = -x_0$ and $\partial d_{\vp_0}/\partial v(0,0) = 0$, the function $d_{\vp_0}$ is singular at $(0,0)$ if and only if $x_0 = 0$.

We assume that $d_{\vp_0}$ is singular at $(0,0)$.
Then
\[
j^2 d_{\vp_0}(0) = \dfrac12\bigl(\|\vp_0\|^2 - ((b_2 y_0 + a_{2,0} z_0 - 1)u^2 + y_0 v^2)\bigr),
\]
and thus $d_{\vp_0}$ has an $A_1$-singularity at $(0,0)$ if and only if $y_0(b_2 y_0 + a_{2,0} z_0 -1) \ne 0$.
Moreover, the singularity of $d_{\vp_0}$ at $(0,0)$ is of type $A_{\geq 2}$ if and only if 
(i) $y_0\ne0$ and $b_2 y_0 + a_{2,0} z_0 -1 = 0$ or
(ii) $y_0 = 0$ and $a_{2,0} z_0 \ne 1$ hold, and that is of type $D_{\geq 4}$ or more degenerate if and only if $y_0 = 0$ and $a_{2,0} z_0 = 1$.

We assume that the condition (i) holds.
Since $y_0 \ne 0$, by replacing $v$ by $v - a_{2,1} z_0 u^2/(2y_0)$, we can reduce 4-jet of $d_{\vp_0}$ to
\begin{align*}
j^4 d_{\vp_0}(0) & = \dfrac12\|\vp_0\|^2 - \dfrac12 y_0 v^2 -\dfrac16((b_3 y_0 + a_{3,0} z_0) u^3 + 3 a_{1,2} z_0 u v^2 + a_{0,3}z_0 v^3)\\
&\quad - \dfrac1{24}\left(\dfrac{b_4 y_0^2 +a_{4,0} y_0 z_0 - 3a_{2,1}^2 z_0^2 -3 y_0 (a_2^2 + b_2^2)}{y_0} + 4 c_{3,1} u^3 v + 6 c_{2,2} u^2 v^2 + 4 c_{1,3} u v^3 + a_{0,4} v^4\right),
\end{align*}
where $c_{3,1}, c_{2,2}, c_{1,3}, c_{0,4} \in \R$.
From this we see that the assertions (2a), (3a) and (4a) hold.

We turn to the case (ii) and assume that condition (ii) holds.
Since $a_{2,0} z_0 \ne 1$, by replacing $u$ by $u - a_{1,2}z_0 (2(a_{2,0}z_0 - 1))^{-1}v^2$, we can reduce 4-jet of $d_{\vp_0}$ to
\begin{align*}
j^4d_{\vp_0}(0) & = \dfrac12 z_0^2 - \dfrac12(a_{2,0}z_0 - 1)u^2 -\dfrac16 z_0(a_{3,0}u^3 + 3 a_{2,1}u^2 v + a_{0,3}v^3) \\
& \quad - \dfrac1{24}\left(\hat c_{4,0}u^4 + 4\hat c_{3,1}u^3 v + 6\hat c_{2,2}u^2 v^2 + 4\hat c_{1,3}u v^3 + \dfrac{(a_{0,4}\,a_{2,0} - 3a_{1,2}^2)z_0^2 - (a_{0,3} + 3a_{2,0})z_0 + 3}{a_{2,0}z_0 - 1}\right),
\end{align*}
where $\hat c_{4,0}, \hat c_{3,1}, \hat c_{2,2}, \hat c_{1,3} \in \R$. 
From this we see that the assertions (2b), (3b) and (4b) hold.

Let us prove the necessary and sufficient conditions for $D$ being an $\mathcal{R}^+$- and $\mathcal{K}$-versal unfolding of $d_{\vp_0}$. 
We skip the proofs of the assertion (1) and (2), because the proofs of (1) and (2) is similar to that of (3).
First, we consider the condition (3a). 
Assume that (3a) holds.
Since $A_3$-singularity is $4$-determined, to see that $D$ is an $\mathcal{R}^+$- or $\mathcal{K}$-versal unfolding of $d_{\vp_0}$ we need to verify the equalities
\begin{align}
\label{eq:DistanceR3a}
\mathcal{E}_2 &= \left\langle\pd{d_{\vp_0}}{u}, \pd{d_{\vp_0}}{v}
\right\rangle_{\mathcal{E}_2} + \left\langle\left.\pd{D}{x}\right|_{\R^2 \times\{\vp_0\}}, \left.\pd{D}{y}\right|_{\R^2 \times\{\vp_0\}}, \left.\pd{D}{z}\right|_{\R^2 \times\{\vp_0\}}\right\rangle_{\R} + \langle 1 \rangle_{\R} + \langle u, v \rangle_{\mathcal{E}_2}^5, \quad \mbox{or}\\
\label{eq:DistanceK3a}
\mathcal{E}_2 &= \left\langle\pd{d_{\vp_0}}{u}, \pd{d_{\vp_0}}{v}, d_{\vp_0} \right\rangle_{\mathcal{E}_2} + \left\langle\left.\pd{D}{x}\right|_{\R^2 \times\{\vp_0\}}, \left.\pd{D}{y}\right|_{\R^2 \times\{\vp_0\}}, \left.\pd{D}{z}\right|_{\R^2 \times\{\vp_0\}}\right\rangle_{\R} + \langle u, v \rangle_{\mathcal{E}_2}^5,
\end{align}
respectively (cf. \cite[p.149]{Mart1982}).
Replacing $v$ by $v - a_{2,1} z_0 u^2/(2 y_0)$, we show that the coefficients of $u^i v^j$ of functions appearing in
 \eqref{eq:DistanceR3a} and \eqref{eq:DistanceK3a} are given by the following tables:
\[
\mbox{\footnotesize{$
\begin{array}{c|c|cc|ccc|cccc|c}
& 1 & u & v & u^2 & u v & v^2 & u^3 & u^2 v & u v^2 & v^3 & u^4 \\\hline
D_x & 0 & \fbox{$-1$} & 0 & 0 & 0 & 0 & 0 & 0 & 0 & 0 & 0\\
D_y & y_0 & 0 & 0 & -\frac{b_2}2 & 0 & -\frac12 & \frac{\alpha_{3,0}}6 & \frac{\alpha_{2,1}}2 & \frac{\alpha_{1,2}}2 & \frac{\alpha_{0,3}}6 & \frac{\alpha_{4,0}}24\\
D_z & z_0 & 0 & 0 & -\frac{a_{2,0}}2 & 0 & 0 & \frac{\beta_{3,0}}6 & \frac{\beta_{2,1}}2 & \frac{\beta_{1,2}}2 & \frac{\beta_{0,3}}6 & \frac{\beta_{4,0}}24\\\hline
(d_{\vp_0})_u & 0 & 0 & 0 & 0 & 0 & -\frac{a_{1,2}z_0}2 & \fbox{$\frac{c_{4,0}}6$} & \frac{c_{3,1}}2 & \frac{c_{2,2}}2 & \frac{c_{1,3}}6 & \frac{c_{5,0}}24 \\
(d_{\vp_0})_v & 0 & 0 & \fbox{$-y_0$} & 0 & -a_{1,2}z_0 & -\frac{a_{0,3}z_0}2 & \frac{c_{3,1}}6 & \frac{c_{2,2}}2 & \frac{c_{1,3}}2 & \frac{c_{0,4}}6 & \frac{c_{4,1}}24 \\\hline
u(d_{\vp_0})_u & 0 & 0 & 0 & 0 & 0 & 0 & 0 & 0 & -\frac{a_{1,2}z_0}2 & 0 & \fbox{$\frac{c_{4,0}}6$} \\
u(d_{\vp_0})_v & 0 & 0 & 0 & 0 & \fbox{$-y_0$} & 0 & 0 & -a_{1,2}z_0 & -\frac{a_{0,3}z_0}2 & 0 & \frac{c_{3,1}}6 \\
v(d_{\vp_0})_v & 0 & 0 & 0 & 0 & 0 & \fbox{$-y_0$} & 0 & 0 & -a_{1,2}z_0 & -\frac{a_{0,3}z_0}2 & 0\\\hline
u^2(d_{\vp_0})_v & 0 & 0 & 0 & 0 & 0 & 0 & 0 & \fbox{$-y_0$} & 0 & 0 & 0\\
u v(d_{\vp_0})_v & 0 & 0 & 0 & 0 & 0 & 0 & 0 & 0 & \fbox{$-y_0$} & 0 & 0\\
v^2(d_{\vp_0})_v & 0 & 0 & 0 & 0 & 0 & 0 & 0 & 0 & 0 & \fbox{$-y_0$} & 0\\
\end{array}$}}
\]
\[
\mbox{\small{$
\begin{array}{c|c|ccccc}
& u^i v^j\ (i + j \leq 3) & u^4 & u^3 v & u^2 v^2 & u v^3 & v^4 \\\hline
u^3 (d_{\vp_0})_v & 0 & 0 & \fbox{$-y_0$} & 0 & 0 & 0 \\
u^2 v (d_{\vp_0})_v & 0 & 0 & 0 & \fbox{$-y_0$} & 0 & 0 \\
u v^2 (d_{\vp_0})_v & 0 & 0 & 0 & 0 & \fbox{$-y_0$} & 0 \\
v^3 (d_{\vp_0})_v & 0 & 0 & 0 & 0 & 0 & \fbox{$-y_0$}\\\hline
\end{array}$}}
\]
Here
\[
c_{i,j} = \dfrac{\partial^{i + j} d_{\vp_0}}{\partial u^i \partial v^j}(0,0),\quad \alpha_{i,j} = \dfrac{\partial^{i + j + 1} D}{\partial u^i \partial v^j \partial y}(0,0,\vp_0)\quad\mbox{and}\quad
 \beta_{i,j} = \dfrac{\partial^{i + j + 1} D}{\partial u^i \partial v^j \partial z}(0,0,\vp_0).
\]
We note that $c_{4,0} = -(b_4 y_0^2 +a_{4,0} y_0 z_0 - 3_{2,1}^2 z_0^2 -3 y_0 (a_2^2 + b_2^2))/y_0\ne0$.
Since boxed entries are non-zero, the matrix represented by the above tables is of full rank, that is, the equality \eqref{eq:DistanceR3a} (resp. \eqref{eq:DistanceK3a}) holds if and only if $(a_{2,0},b_2) \ne (0,0)$ (resp. $b_2 y_0 + a_{2,0} z_0 \ne0$).
However, since now $b_2 y_0 + a_{2,0} z_0 - 1 = 0$ holds, we have $(a_{2,0},b_2) \ne (0,0)$.
Therefore, if (3a) holds, then $D$ is an $\mathcal{R}^+$-versal unfolding of $d_{p_0}$. 
 
Next, we assume that (3b) holds. 
Similar to (3a), we need to verify \eqref{eq:DistanceR3a} or \eqref{eq:DistanceK3a} holds.
Since 
\begin{align*}
\pd{D}{x}(u,v,\vp_0) & = -u,\quad \pd{D}{y}(u,v,\vp_0) = -\dfrac12(b_2^2 u^2 + v^2) + O(u,v)^3,\\
\pd{D}{z}(u,v,\vp_0) & = z_0 - \dfrac12a_{2,0} u^2 + O(u,v)^3,\\
\pd{d_{\vp_0}}{u}(u,v) & = (1 - a_{2,0}z_0) u - \dfrac12(2a_{2,1}z_0 u v + a_{1,2} v^2) + O(u,v)^3,\quad \\
\pd{d_{\vp_0}}{v}(u,v) & = -\dfrac12(a_{2,1}z_0 u^2 + 2 a_{1,2}z_0 u v + a_{0,3} z_0 v^2) + O(u,v)^3, 
\end{align*}
neither \eqref{eq:DistanceR3a} nor \eqref{eq:DistanceK3a} does not hold. 

Now we turn to prove (4). 
The number of parameters in an $\mathcal{R}^+$-mini-versal unfolding of $A_5$-singularity is 4.
Therefore, $D$ is not an $\mathcal{R}^+$-versal unfolding of $d_{\vp_0}$ having $A_{\geq 5}$-singularity because it is a 3-parameter unfolding. 
For the similar reason, $D$ is not an $\mathcal{K}$-versal unfolding of $d_{\vp_0}$ having $A_{\geq4}$-singularity.

We assume that (4a) holds and $d_{\vp_0}$ has an $A_4$-singularity at $(0,0)$. 
Since $A_4$-singularity is 5-determined, to see that $D$ is an $\mathcal{R}^+$-versal unfolding of $d_{\vp_0}$ we need to verify the equality
\begin{align}
\label{eq:DistanceR4a}
\mathcal{E}_2 &= \left\langle\pd{d_{\vp_0}}{u}, \pd{d_{\vp_0}}{v}\right\rangle_{\mathcal{E}_2} + \left\langle\left.\pd{D}{x}\right|_{\R^2 \times\{\vp_0\}}, \left.\pd{D}{y}\right|_{\R^2 \times\{\vp_0\}}, \left.\pd{D}{z}\right|_{\R^2 \times\{\vp_0\}}\right\rangle_{\R} + \langle 1 \rangle_{\R} + \langle u, v \rangle_{\mathcal{E}_2}^6.
\end{align}
We consider the table in the proof of (3a).
Since, $d_{\vp_0}$ has an $A_4$-singularity at $(0,0)$, we have $c_{4,0} = 0$ and $c_{5,0} \ne0$. 
Hence, \eqref{eq:DistanceR4a} holds if and only if
\[
\begin{vmatrix}
-\dfrac{b_2}2 & \dfrac{\alpha_{3,0}}6 \\
-\dfrac{a_{2,0}}2 & \dfrac{\beta_{3,0}}6
\end{vmatrix}
= -\dfrac1{12}(a_{2,0} b_3 - a_{3,0} b_2) \ne0.
\]
Let denote $L_1$ and $L_2$, respectively, lines $b_2 y + a_{2,0} z -1 = 0$ and $b_3 y + a_{3,0} z = 0$ on the $yz$-plane. 
We remark that now $(a_{2,0}, b_2)\ne(0,0)$ holds by the same reason as in (3a). 
The condition that $b_2 y_0 + a_{2,0} z_0 -1 = b_3 y_0 + a_{3,0} z_0 = 0$ is equivalent to the condition that a point $(y_0,z_0)$, on the $yz$-plane, is the intersection of $L_1$ and $L_2$ or is on $L_1$ when $(a_{3,0},b_3)\ne(0,0)$ or $(a_{3,0},b_3)=(0,0)$, respectively. 
Therefore, if $(a_{3,0},b_3)\ne(0,0)$ (resp. $=(0,0)$) then $a_{2,0}b_3 - a_{3,0}b_2 \ne0$ (resp. $=0$), and vice versa. 
Remark that $a_{3,0} = b_3 = 0$ if and only if $d_{\vp_0}$ has $A_{\geq 3}$-singularity at $(0,0)$ for any $\vp_0\in L_1$.
 
If (4b) holds, then $D$ is not an $\mathcal{R}^+$-versal unfolding of $d_{\vp_0}$ having an $A_4$-singularity at $(0,0)$ by the same reason as in (3b). 

Now, we shall prove (5). 
Since the number of parameters in an $\mathcal{R}^+$- (resp. $\mathcal{K}$-) mini-versal unfolding of $D_5$ (resp. $D_4$) is 4, $D$ is not an $\mathcal{R}^+$- (resp. $\mathcal{K}$-) versal unfolding of $d_{\vp_0}$ having a $D_5$ (resp. $D_4$) or more degenerate singularity at $(0,0)$. 
Moreover, since $D_4$-singularity is 3-determined, $D$ is an $\mathcal{R}^+$-versal unfolding of $d_{\vp_0}$ having a $D_4$-singularity at $(0,0)$ if and only if 
\begin{align}
\label{eq:DistanceR5}
\mathcal{E}_2 &= \left\langle\pd{d_{\vp_0}}{u}, \pd{d_{\vp_0}}{v}\right\rangle_{\mathcal{E}_2} + \left\langle\left.\pd{D}{x}\right|_{\R^2 \times\{\vp_0\}}, \left.\pd{D}{y}\right|_{\R^2 \times\{\vp_0\}}, \left.\pd{D}{z}\right|_{\R^2 \times\{\vp_0\}}\right\rangle_{\R} + \langle 1 \rangle_{\R} + \langle u, v \rangle_{\mathcal{E}_2}^4.
\end{align}
holds.
If $d_{p_0}$ has a $D_4$ singularity at $(0,0)$, then 
\begin{align*}
\pd{D}{x}(u,v,\vp_0) & = -u,\quad \pd{D}{y}(u,v,\vp_0) = -\dfrac12(b_2^2 u^2 + v^2) + O(u,v)^3,\\
\pd{D}{z}(u,v,\vp_0) & = \dfrac1{a_{2,0}} - \dfrac12a_{2,0} u^2 + O(u,v)^3,\\
\pd{d_{\vp_0}}{u}(u,v) & = -\dfrac1{2a_{2,0}}(a_{3,0} u^2 + 2a_{2,1} u v + a_{1,2}v^2) + O(u,v)^3,\quad\\
\pd{d_{\vp_0}}{v}(u,v) & = -\dfrac1{2a_{2,0}}(a_{2,1} u^2 2a_{1,2} u v + a_{0,3} v^2) + O(u,v)^3,  
\end{align*}
and thus \eqref{eq:DistanceR5} does not hold.
\end{proof}

Let $g:(\R^2, 0) \to (\R^3, 0)$ be a smooth map-germ of corank 1 at the origin, and let $j^2g(0)$ be $\mathcal{A}$-equivalent to $(u,v^2,0)$.
From Proposition \ref{prop:distance}, the locus of points $\vp$ where $d_{\vp}$ has a degenerate singularity at $(0,0)$ consists of one or two lines on the normal plane. 
We call such a locus a \textit{focal locus}.
The focal locus contains the principal normal line.  
The focal locus can be considered as an analogy of the focal conic of Whitney umbrellas (cf. \cite[Lemma 3.3]{FH2012-1}).
From Proposition \ref{prop:distance} and definitions of inflection and degenerate inflection points, we can easily show that the following proposition holds: 
\begin{proposition}
The focal locus is, 
\begin{enumerate}
\item 
a pair of two intersecting lines if and only if the origin is not an inflection point,
\item 
a pair of two parallel lines if and only if the origin is a non-degenerate inflection point,
\item 
the principal normal line if and only if the origin is a degenerate inflection point.
\end{enumerate}
In case $(1)$, the distance squared function $d_{\vp}$ has a singularity of type $D_4$ or more degenerate at $(0,0)$ if and only if $\vp$ is the intersection point of these two lines. 
This is related to the umbilic curvature introduced in \cite{MN-B2015}. 
\end{proposition}

\begin{proof}[Proof of Theorem \ref{thm:Distance}]
Firstly, we remark that the condition \eqref{eq:assumption} and the following condition hold.
\[
(x_0,y_0,z_0) = \lambda\left(0,\,-\dfrac{a_{n+1,1}\cos\theta_0}{\ma(\theta_0)}\,\dfrac{(n+1)!\sin\theta_0}{\ma(\theta_0)}\right)\quad(\lambda \ne0).
\]

(1) We skip the proofs of (1a), (1b) and (1d) because the proofs are similar to that of (1c).
We will only prove (1c).
From the assumption, now we have
\[
y_0\ne0\quad\text{and}\quad a_{n+1,1}\,b_2 \cos\theta_0 - (n+1)!\,a_{2,0}\sin\theta_0 \ne 0.
\]
Since
\[
b_2 y_0 + a_{2,0} z_0 -1 = \dfrac{\lambda(-a_{n+1,1}\,b_2 \cos\theta_0 + (n+1)!\,a_{2,0}\sin\theta_0)}{\ma(\theta_0)}-1 = 0,
\]
we obtain
\[
\lambda = \dfrac{\ma(\theta_0)}{-a_{n+1,1}\,b_2 \cos\theta_0 + (n+1)!\,a_{2,0}\sin\theta_0} = \dfrac1{\tk{1}(0,\theta_0)}.
\]
Then we obtain
\begin{align*}
& b_3 y_0 + a_{3,0} z_0 =  \dfrac{-a_{n+1,1}\,b_3 \cos\theta_0 + (n+1)!\,a_{3,0} \sin\theta_0}{\tk{1}(0,\theta_0)\ma(\theta_0)} = -\dfrac{\Delta_1^{n+1}(\theta_0)}{\tk{1}(\theta_0)\ma(\theta_0)}\\
\intertext{and}
& b_4 y_0^2 + a_{4,0} y_0 z_0 - 3 a_{2,1}^2 z_0^2 - 3(a_{2,0}^2 + b_2^2) y_0\\
& = \dfrac1{\tk{1}(\theta_0)\ma(\theta_0)}\Bigl(a_{n+1,1}\bigl(a_{n+1,1}\,b_4\cos\theta_0 - (n+1)!\,a_{4,0}\sin\theta_0 \bigr)\cos\theta_0 - 12a_{2,1}^2\sin^2\theta_0\\
& \quad - 3 a_{n+1,1}(a_{2,0}^2 + b_2^2)\bigl(a_{n+1,1}\,b_2\cos\theta_0 - (n+1)!\,a_{2,0}\sin\theta_0\bigr)\cos\theta_0\Bigr)\\
& = -\dfrac{a_{n+1,1}\,\Delta_2^{n+1}(\theta_0)}{\tk{1}(\theta_0)\ma(\theta_0)}.
\end{align*}
Therefore, from Proposition \ref{prop:distance}, we conclude that $(0,\theta_0)$ is a first order ridge point relative to $\tv{1}$ over the singularity if and only if $d_{\vp_0}$ has an $A_3$-singularity at $(0,0)$. 
Moreover, since
\[
a_{2,0} y_0 + b_2 z_0 = \dfrac{-a_{n+1,1}\,a_{2,0}\cos\theta_0 + (n+1)!\,b_2\sin\theta_0}{\tk{1}(\theta_0)\ma(\theta_0)} = -\dfrac{\Delta_3^{(n+1)}(\theta_0)}{\tk{1}(\theta_0)\ma(\theta_0)},
\]
$D$ is a $\mathcal{K}$-versal unfolding of $d_{\vp_0}$ if and only if $(0,\theta_0)$ is not a sub-parabolic point relative to $\tv{2}$ over singularity. 

(2) The statements follow immediately from the definition of the focal locus and Proposition \ref{prop:distance}. 
\end{proof}

\section{Wave-fronts and caustics of singular surfaces.}
The wave-front or parallel of a surface in $\R^3$ is the envelope of spheres with the centers on the surface.
On the other hand, the caustic of the surface is the envelope of normal rays to the surface.
It is also the locus of the singular points on the wave-front of the surface.

We define a family of functions $\wtilde D$ on a surface parameterized by a smooth map-germ $g:(\R^2,0)\to(\R^3,0)$ by
\[
\wtilde D:(\R^2,0)\times(\R^3,\vp_0)\to\R,\quad \tilde D(u,v,\vp) = \dfrac12(\|g(u,v)-\vp\| - t_0^2),
\]
where $t_0$ is a non-negative constant.
We define $\wtilde d_{\vp_0}(u,v) = \wtilde D(u,v,\vp_0)$.
The discriminant set of $\wtilde D$ is given by
\[
\mathcal{D}(\wtilde D) = \{\vp\in(\R^3,0)\,|\,\wtilde D = \wtilde D_u = \wtilde D_v = 0\mbox{ for some }(u,v)\in(\R^2,0)\},
\]
which is the wave-front of the surface at a distance $\pm t_0$.
On the other hand, the bifurcation set of $D$ is given by
\[
\mathcal{B}(D) = \{\vp\in(\R^3,0)\,|\,D_u = D_v = D_{u u} D_{v v} - D_{u v}^2 = 0\mbox{ for some }(u,v)\in(\R^2,0)\}, 
\]
which is the caustic of the surface.
It is well-known (see \cite[Theorem 3.4]{Wall1981}, for example) that two $\mathcal{K}$ (resp. $\mathcal{R}^+$)-versal unfoldings of $d_{\vp}$ are $\mathcal{K}$ (resp. $\mathcal{R}^+$)-isomorphic as unfoldings. 
Therefore, when $D$ is $\mathcal{K}$ (resp. $\mathcal{R}^+$)-versal, we can conclude the diffeomorphic types of the wave-fronts (resp. caustics) of our singular surfaces $S$. 
Let us state the simplest case of the conclusion of Theorem 3.1 as following \textR{theorem}:  
\begin{theorem}
\label{cor:wave}
Let $S$ be a singular surface parameterized by $g:(\R^2,0)\to(\R^3,0)$ in the form \eqref{eq:normal_form_corank1}, and let $g$ be $\ma$-equivalent to one of $S_k$, $B_k$, $C_k$ and $F_4$ singularities.
Suppose that $\wtilde\kappa_1(0,\theta_0)\ne0$ and $\vp_0 = \wtilde{\bm{n}}(0,\theta_0)/\wtilde\kappa_1(0,\theta_0)$ where $\theta_0\in(-\pi/2,\pi/2)$. 
\begin{enumerate}
\item
If $(0,\theta_0)$ is not a ridge point relative to $\wtilde{\bm{v}}_1$ over the singularity of $S$, the singularity of the wave-front of $S$ at $\vp_0$ is a cuspidal edge. 
\item 
If $(0,\theta_0)$ is a first order ridge relative to $\wtilde{\bm{v}}_1$ but not a sub-parabolic point relative to $\wtilde{\bm{v}}_2$ over the singularity of $S$, then the singularity of the wave-front and caustic of $S$ at $\vp_0$ is a swallowtail and cuspidal edge, respectively. 
\end{enumerate}
\end{theorem}
Here, a singularity is called a \textit{cuspidal edge} or \textit{swallowtail} if the corresponding map-germs is $\mathcal{A}$-equivalent to
\[
f_c:=(u^2, u^3, v)\quad \mbox{or}\quad f_s:=(3 u^4 + u^2 v, 4 u^3 + 2 u v, v),
\]
respectively. 

\appendix
\section{Coefficients.}
\begin{align}
\label{eq:blow_n21}
n_{21} & = - \frac{((n+1)!)^2(a_{n+2,1}\cos\theta + (n+2)!\,a_{1,2}\sin\theta)\cos\theta\sin^2\theta}{(n+2)\ma(\theta)^3},\\
n_{31} & = -\frac{(n+1)!\,a_{n+1,1}(a_{n+2,1}\cos\theta + (n+2)!\,a_{1,2}\sin\theta)\cos^2\theta\sin\theta}{(n+2)\ma(\theta)^3}.
\end{align}
\begin{align}
\begin{split}
\label{eq:blow_n22}
n_{22} & =\text{\footnotesize{$\frac1{\ma(\theta)^5}\Biggl[\frac{{a_{n+1,1}}^5\, b_2^2\cos^5\theta}2 - (n+1)!\, a_{n+1,1}^4\,a_{2,0}\,b_2\cos^4\theta\sin\theta$}}\\
&\quad\text{\footnotesize{$-((n+1)!)^2\,a_{n+1,1}
\Biggl(\dfrac{a_{n+3,1}\,a_{n+1,1}}{(n+3)(n+2)} - \dfrac{3a_{n+2,1}^2}{2(n+2)^2} - \dfrac{a_{n+1,1}^2(a_{2,0}^2 + b_2^2)}2\Biggr)\cos^3\theta\sin^2\theta$}}\\
&\quad\text{\footnotesize{$+((n+1)!)^3\,a_{n+1,1}\Biggl(\dfrac{a_{n+2,1}\,a_{1,2}}{n+2} - \dfrac{a_{n+1,1}(a_{2,0}\,b_2 - a_{2,2})}2\Biggr)\cos^2\theta\sin^3\theta$}}\\
&\quad\text{\footnotesize{$- ((n+1)!)^4\Biggl(\dfrac{a_{n+3,1}}{(n+3)(n+2)} - \dfrac{a_{n+1,1}(3a_{1,2}^2 + a_{2,0}^2)}2
\Biggr)\cos\theta\sin^4\theta -\dfrac{((n+1)!)^5\,a_{2,2}}2\sin^5\theta\Biggr]\cos^2\theta$}}\\
&\quad\text{\footnotesize{$ - \epsilon\dfrac{4 a_{0,3}\bigl(a_{2,1}^2\cos^2\theta + 4 \sin^2\theta\bigr)\cos\theta\sin^4\theta}{\ma(\theta)^5}
$}},
\end{split}\\
\begin{split}
\label{eq:blow_n32}
n_{32} & =\text{\footnotesize{$\dfrac1{\ma(\theta)^5}\Biggl[ - (n+1)!\,a_{n+1,1}^2\Biggl(\dfrac{a_{n+3,1}\,a_{n+1,1}}{(n+3)(n+2)}
   - \dfrac{a_{n+2,1}^2}{(n+2)^2}
   + \dfrac{a_{n+1,1}^2\,b_2}2
   \Biggr)\cos^4\theta$}}\\
   &\quad\text{\footnotesize{$
   + ((n+1)!)^2\,a_{n+1,1}^2
   \Biggl(
   \dfrac{2a_{n+2,1}\,a_{1,2}}{n+2}
   + \dfrac{a_{n+1,1}(2a_{2,0}\,b_2 - a_{2,2})}2
   \Biggr)\cos^3\theta\sin\theta$}}\\
   &\quad\text{\footnotesize{$
   - ((n+1)!)^3
   \Biggl(
   \dfrac{a_{n+3,1}\,a_{n+1,1}}{(n+3)(n+2)}
   + \dfrac{a_{n+2,1}^2}{2(n+2)^2}
   - \dfrac{a_{n+1,1}^2(2a_{1,2}^2 - a_{2,0}^2 - b_2^2)}2
   \Biggr)\cos^2\theta\sin^2\theta$}}\\
   &\quad\text{\footnotesize{$
   - ((n+1)!)^4
   \Biggl(
   \dfrac{a_{n+2,1}\,a_{1,2}}{n+2}
   - \dfrac{a_{n+1,1}(2a_{2,0}\,b_2 - a_{2,2})}2
   \Biggr)\cos\theta\sin^3\theta$}}\\
   &\quad\text{\footnotesize{$
   - \dfrac{((n+1)!)^5(a_{1,2}^2 + a_{2,0}^3)\sin^4\theta}2
   \Biggr]\cos^2\theta\sin\theta
   - 
   \dfrac{2 a_{2,1}\,a_{0,3}
   (a_{2,1}^2 \cos^2\theta + 4\sin^2\theta)\cos^2\theta\sin^3\theta}{\ma(\theta)^5}
   $}}.
  \end{split}
 \end{align}

 \begin{align}
  \label{eq:blow_L1}
  \begin{split}
   \
   L_1 &
   =
   \text{\footnotesize{$
   \frac{(-a_{n+1,1}\cdot
   b_3\cos\theta+(n+1)!a_{30}\sin\theta)\cos\theta}{\ma(\theta)}$}}\\
   &\text{\footnotesize{$
   \quad + \dfrac{(n+1)!}{(n+2)\ma(\theta)^3}
   \left(
   \begin{array}{l}
    a_{n+2,1}\,a_{n+1,1}\,a_{2,0} \cos^2\theta\\
     + (n+1)!\bigl((n+2)a_{n+1,1}\,a_{1,2}\,a_{2,0} +
     a_{n+2,1}\,b_2\bigr)\cos\theta\sin\theta\\
    + (n+2)((n+1)!)^2 a_{1,2}\,b_2 \sin^2\theta
   \end{array}
   \right)\cos\theta\sin\theta$}}
  \end{split}
 \end{align}

 \begin{align}
  \begin{split}
   M_1 &
  =
  \text{\footnotesize{$
  \dfrac1{\ma(\theta)}\bigl(a_{n+2,1} \cos\theta +
  (n+1)!\,a_{1,2}\sin\theta\bigr)$}}\\
  & \quad
  \text{\footnotesize{$
  -
  \dfrac{(n+1)a_{n+1,1}^2}{(n+2)\ma(\theta)^3}
  \bigl(a_{n+2,1} \cos\theta + (n+2)!\,a_{1,2}\sin\theta
  \bigr)\cos^{n+1}\theta\sin\theta
  $}}
  \end{split}
 \end{align}
 
 \begin{align}
  N_1 &
  =
  \text{\footnotesize$
  \frac{(n+1)!\,a_{1,2}\cos\theta\sin\theta}{\ma(\theta)}
  -\frac{((n+1)!)^2}{(n+2)\ma(\theta)^3}
  \bigl(
  a_{n+2,1}\cos\theta + (n+2)!\,a_{1,2}\sin\theta
  \bigr)\cos\theta\sin^2\theta,$}
 \end{align}

 \begin{align}
  \begin{split}
   \label{eq:pullback_L2}
   L_2 &
   =
   \text{\footnotesize$
   \frac{(-a_{n+1,1}\,b_4\cos\theta +
   (n+1)!\,a_{4,0}\sin\theta)\cos^2\theta}{2\ma(\theta)}
   - \frac{(n+1)!}{\ma(\theta)^3}
   \Biggl[
   \frac{a_{n+2,1}\,a_{n+1,1}\,a_{3,0}}{n+2}\cos^2\theta$}\\
   &\quad\text{\footnotesize$
   + (n+1)!
   \left(
   \frac{a_{n+2,1}\,b_3}{n+2} + a_{n+1,1}\,a_{3,0}\,a_{1,2}
   \right)\cos\theta\sin\theta
   + ((n+1)!)^2 a_{1,2}\,b_3\sin^2\theta
   \Biggr]\cos^2\theta\sin\theta$}\\
   &\quad\text{\footnotesize$
   +\frac1{\ma(\theta)^5}
   \Biggl[
   \frac{{a_{n+1,1}}^5\,b_2^3}2\cos^5\theta
   - \frac{(n+1)!\,a_{n+1,1}^2\,a_{2,0}}2
   \Biggl(
   \frac{2a_{n+3,1}\,a_{n+1,1}}{(n+3)(n+2)}
   - \frac{2a_{n+2,1}^2}{(n+2)^2}$}\\
   &\quad\text{\footnotesize$
   +3 a_{n+1,1}^2\,b_2^2
   \Biggr)\cos^4\theta\sin\theta
   - ((n+1)!)^2 a_{n+1,1}
   \Biggl(
   \frac{a_{n+3,1}\,a_{n+1,1}\,b_2}{(n+3)(n+2)}
   -\frac{3a_{n+2,1}^2\,b_2}{2(n+2)^2}$}\\
   &\quad\text{\footnotesize$
   -\frac{2a_{n+2,1}\,a_{n+1,1}\,a_{1,2}\,a_{2,0}}{n+2}
   +\frac{a_{n+1,1}^2(a_{2,2}\,a_{2,0} -3a_{2,0}^2\,b_2 - b_2^3)}2
   \Biggr)\cos^3\theta\sin^2\theta$}\\
   &\quad\text{\footnotesize$
   - ((n+1)!)^3
   \Biggl(
   \frac{a_{n+3,1}\, a_{n+1,1}\, a_{2,0}}{(n+3)(n+2)}
   + \frac{a_{n+2,1}^2\, a_{2,0}}{2(n+2)^2}
   - \frac{3a_{n+2,1}\, a_{n+1,1}\, a_{1,2}\,b_2}{n+2}$}\\
   &\quad\text{\footnotesize$
   + \frac{a_{n+1,1}^2(a_{2,2}\,b_2 - 2 a_{1,2}^2\, a_{2,0} + a_{20}^3 - 3
   a_{2,0}\, b_2^2)}2
   \Biggr)\cos^2\theta\sin^3\theta
   - ((n+1)!)^4
   \Biggl(
   \frac{a_{n+3,1}\, b_2}{(n+3)(n+2)}$}\\
   &\quad\text{\footnotesize$
   + \frac{a_{n+2,1}\, a_{1,2}\, a_{2,0}}{n+2}
   + \frac{a_{n+1,1}(a_{2,2}\, a_{2,0} - 3 a_{1,2}^2\, b_2 - 3
   a_{2,0}^2\, b_2)}2
   \Biggr)\cos\theta\sin^4\theta$}\\
   &\quad\text{\footnotesize$
   -\frac{((n+1)!)^5(a_{2,2}\, b_2 + a_{1,2}^2\, a_{2,0} +
   a_{2,0}^3)\sin^5\theta}2
   \Biggr]\cos^2\theta
   + \epsilon
   \Biggl[
   \dfrac{2a_{2,1}\cos\theta\sin^2\theta}{\ma(\theta)}$}\\
   &\quad\text{\footnotesize$
   -\dfrac{2a_{0,3}}{\ma(\theta)^5}
   \Biggl(
   a_{2,1}^3\,a_{2,0}\cos^3\theta
   + 2 a_{2,1}^2\,b_2 \cos^2\theta\sin\theta
   + 4 a_{2,1}\,a_{2,0}\cos\theta\sin^2\theta$}\\
   &\quad\text{\footnotesize$
   + 8b_2 \sin^3\theta\Biggr)\cos\theta\sin^3\theta
   \Biggr]$},
  \end{split}
 \end{align}

 \begin{align}
  \begin{split}
   \label{eq:pullback_M2}
   M_2 &
   =
   \text{\footnotesize$
   \dfrac1{\ma(\theta)}
   \left(
   \dfrac{a_{n+3,1}\cos\theta}{n+2}
   + (n+1)!\,a_{2,2}\sin\theta
   \right)\cos^{n+1}\theta\sin\theta
      -\dfrac{a_{n+1,1}}{\ma(\theta)^3}
   \Biggl[
   \dfrac{a_{n+2,1}^2\cos^2\theta}{n+2}$}\\
   &\quad\text{\footnotesize$
   + \dfrac{(n+3)(n+1)!\,a_{n+2,1}\,a_{1,2}\cos\theta\sin\theta}{n+2}
   + ((n+1)!)^2\,a_{1,2}^2\sin^2\theta
   \Biggr]\cos^{n+2}\theta\sin\theta$}\\
   &\quad\text{\footnotesize$
   - \dfrac{(n+1)a_{n+1,1}}{\ma(\theta)^5}
   \Biggl[
   a_{n+1,1}^2
   \left(
   \dfrac{a_{n+3,1}\,a_{n+1,1}}{(n+3)(n+2)}
   - \dfrac{a_{n+2,1}^2}{(n+2)^2}
   + \dfrac{a_{n+1,1}\,b_2^2}2
   \right)\cos^4\theta$}\\
   &\quad\text{\footnotesize$
   -(n+1)!\,a_{n+1,1}^2
   \left(
   \dfrac{2a_{n+2,1}\,a_{1,2}}{n+2}
   - \dfrac{a_{n+1,1}\,a_{2,2}}2
   + a_{n+1,1}\,a_{2,0}\,b_2
   \right)\cos^3\theta\sin\theta$}\\
   &\quad\text{\footnotesize$
   +((n+1)!)^2
   \left(
   \dfrac{a_{n+3,1}\,a_{n+1,1}}{(n+3)(n+2)}
   + \dfrac{a_{n+2,1}^2}{2(n+2)^2}
   - \dfrac{a_{n+1,1}^2(2a_{1,2}^2 - a_{2,0}^2 - b_2^2)}2
   \right)\cos^2\theta\sin^2\theta$}\\
   &\quad\text{\footnotesize$
   +((n+1)!)^3
   \left(
   \dfrac{a_{n+2,1}\,a_{1,2}}{n+2}
   + \dfrac{a_{n+1,1}(a_{2,2} - 2 a_{2,0}\,b_2)}2
   \right)\cos\theta\sin^3\theta$}\\
   &\quad\text{\footnotesize$
   + \dfrac{((n+1)!)^4\,a_{n+1,1}(a_{1,2}^2 + a_{2,0}^2)\sin^4\theta}2
   \Biggr]\cos^{n+2}\theta\sin\theta$}\\
   &\quad\text{\footnotesize$
   - \dfrac{2a_{2,1}^2\,a_{0,3}(a_{2,1}^2\cos^2\theta +
   4\sin^2\theta)\cos^3\theta\sin^3\theta}{\ma(\theta)^5}$},
  \end{split}
 \end{align}
 
 \begin{align}
  \begin{split}
  \label{eq:blow_N2}
   N_2 &
   =
   \text{\footnotesize$
   \frac{(n+1)!\,a_{2,2}\cos^2\theta\sin\theta}{2\ma(\theta)}
   - \frac{(n+1)!\,a_{n+1,1}\,a_{1,2}}{\ma(\theta)^3}
   \Biggl(
   \frac{a_{n+2,1}\cos\theta}{n+2} + (n+1)!\,a_{1,2}\sin\theta
   \Biggr)\cos^3\theta\sin\theta$}\\
   &\quad\text{\footnotesize$
   + \frac1{\ma(\theta)^5}\biggl[
   \frac{a_{n+1,1}^5\, b_2^2\cos^5\theta}2
   - (n+1)!\, a_{n+1,1}^4\, a_{2,0}\,b_2\cos^4\theta\sin\theta$}\\
   &\quad\text{\footnotesize$
   + ((n+1)!)^2\,a_{n+1,1}
   \Biggl(
   -\frac{a_{n+3,1}\, a_{n+1,1}}{(n+3)(n+2)}
   +\frac{3 a_{n+2,1}^2}{2(n+2)^2}
   +\frac{a_{n+1,1}^2\, a_{2,0}^2}2
   +\frac{a_{n+1,1}^2\, b_2^2}2
   \Biggr)\cos^3\theta\sin^2\theta$}\\
   &\quad\text{\footnotesize$
   + ((n+1)!)^3\,a_{n+1,1}
   \Biggl(
   \frac{3a_{n+2,1}\, a_{1,2}}{n+2}
   - \frac{a_{n+1,1}\, a_{2,2}}2
   - a_{n+1,1}\, a_{2,0}\,b_2
   \Biggr)\cos^2\theta\sin^3\theta$}\\
   &\quad\text{\footnotesize$
   - ((n+1)!)^4
   \Biggl(
   \frac{a_{n+3,1}}{(n+3)(n+2)}
   - \frac{3a_{n+1,1}\, a_{1,2}^2}2
   - \frac{a_{n+1,1}\, a_{2,0}^2}2
   \Biggr)\cos\theta\sin^4\theta$}\\
   &\quad\text{\footnotesize$
   - \frac{((n+1)!)^5\,a_{2,2}\sin^5\theta}2\biggr]\cos^2\theta$}\\
   &\quad\text{\footnotesize$
   + \varepsilon
   \Biggl(
   \dfrac{2 a_{0,3}\cos\theta\sin^2\theta}{\ma(\theta)}
   - \dfrac4{\ma(\theta)^5}
   (4a_{0,3}\sin^2\theta + a_{2,1}^2\,a_{0,3}\cos^2\theta)
   \Biggr)\cos\theta\sin^4\theta$}
 \end{split}
 \end{align}

 \begin{align}
  \begin{split}
   \label{eq:blow_k11}
   k_{11} &
   =
   \text{\footnotesize$
   \frac{(-a_{n+1,1}\, b_3\cos\theta +
   (n+1)!\,a_{3,0}\sin\theta)\cos\theta}{\ma(\theta)} 
   - \frac{(n+1)!}{\ma(\theta)^3}
   \Biggl[
   \frac{a_{n+2,1}\, a_{n+1,1}\, a_{2,0}\cos^2\theta}{n+2}$}\\
   &\quad\text{\footnotesize$
   + (n+1)!
   \Biggl(a_{n+1,1}\, a_{1,2}\,a_{2,0} + \frac{a_{n+2,1}\, b_2}{n+2}
   \Biggr)\cos\theta\sin\theta
   + ((n+1)!)^2\,a_{1,2}\,b_2\sin^2\theta
   \Biggr]\cos\theta\sin\theta$},
  \end{split}
  \end{align}
  
  \begin{align}
  \begin{split}
   \label{eq:blow_k12}
   k_{12} &
   =
   \text{\footnotesize$
   \frac1{2\ma(\theta)}
   \Biggl(
   a_{n+1,1}(2a_{2,0}^2 b_2  + 2b_2^3 - b_4)\cos\theta 
   - (n+1)!(2a_{2,0}\,b_2^2 + 2a_{2,0}^3 - a_{4,0})\sin\theta
   \Biggr)\cos^2\theta$}\\
   &\quad\text{\footnotesize$
   -\frac{(n+1)!}{\ma(\theta)^3}
   \Biggl[
   \frac{a_{n+2,1}\, a_{n+1,1}\, a_{3,0} \cos^2\theta}{n+2}
   + (n+1)!
   \left(
   a_{n+1,1}\,a_{3,0}\,a_{1,2} + \dfrac{a_{n+2,1}\,b_3}{n+2}
   \right)\cos\theta\sin\theta$}\\
   &\quad\text{\footnotesize$
   + ((n+1)!)^2 a_{n+1,1}\, a_{3,0}\, a_{1,2} \sin^2\theta
   \Biggr]\cos^2\theta\sin\theta 
   + \frac1{\ma(\theta)^5}\Biggl[
   \frac{a_{n+1,1}^5\, b_2^3\cos^5\theta}2$}\\
   &\quad\text{\footnotesize$
   - (n+1)!\,a_{n+1,1}^2\, a_{2,0}
   \Biggl(
   \frac{a_{n+3,1}\, a_{n+1,1}}{(n+3)(n+2)}
   - \frac{a_{n+2,1}^2}{(n+2)^2}
   + \frac{3a_{n+1,1}\, b_2^2}2
   \Biggr)\cos^4\theta\sin\theta$}\\
   &\quad\text{\footnotesize$
   - ((n+1)!)^2 a_{n+1,1}
   \Biggl(
   \frac{a_{n+3,1}\, a_{n+1,1}\, b_2}{(n+3)(n+2)}
   - \frac{3 a_{n+2,1}^2\, b_2}{2(n+2)^2}
   - \frac{2a_{n+2,1}\, a_{n+1,1}\, a_{1,2}\, a_{2,0}}{n+2}$}\\
   &\quad\text{\footnotesize$
   - \frac{a_{n+1}^2(3a_{2,0}^2\,b_2 + b_2^3 - a_{2,2}\,a_{2,0})}2
   \Biggr)\cos^3\theta\sin^2\theta
   - ((n+1)!)^3
   \Biggl(
   \frac{a_{n+3,1}\, a_{n+1,1}\, a_{2,0}}{(n+3)(n+2)}
   - \frac{a_{n+2,1}^2\, a_{2,0}}{2(n+2)}$}\\
   &\quad\text{\footnotesize$
   - \frac{3a_{n+2,1}\, a_{n+1,1}\, a_{1,2}\,b_2}{n+2}
   - \frac{a_{n+1,1}^2 (2a_{1,2}^2\, a_{2,0} - a_{2,0}^3 - 3
   a_{2,0}\,b_2^2 - a_{2,2}\,b_2)}2
   \biggr)\cos^2\theta\sin^3\theta$}\\
   &\quad\text{\footnotesize$
   - ((n+1)!)^4
   \Biggl(
   \frac{a_{n+3,1}\, b_2}{(n+3)(n+2)}
   + \frac{a_{n+2,1}\, a_{1,2}\,a_{2,0}}{n+2}
   - \frac{a_{n+1,1}(3a_{1,2}^2\,b_2 + 3 a_{2,0}^2\,b_2 -
   a_{2,2}\,a_{2,0})}2 
   \Biggr)\cos\theta\sin^4\theta$}\\
   &\quad\text{\footnotesize$
   -\frac{((n+1)!)^5
   (a_{1,2}^2\,a_{2,0} + a_{2,2}\,b_2 + a_{2,0}^3)\theta\sin^5\theta}2 
   \Biggr]\cos^2\theta
   + \varepsilon
   \Biggl[
   \dfrac{6a_{2,1}\cos\theta\sin^2\theta}{\ma(\theta)}$}\\
   &\quad\text{\footnotesize$
   - \dfrac{a_{0,3}}{\ma(\theta)^5}
   \Biggl(
   a_{2,1}^3\,a_{2,0}\cos^3\theta
   + 2a_{2,1}^2\,b_2\cos^2\theta\sin\theta
   + 4a_{2,1}\,a_{2,0}\cos\theta\sin^2\theta
   + 8b_2 \sin^3\theta
   \Biggr)\cos\theta\sin^3\theta
   \Biggr]$},
  \end{split}\\
  \begin{split}
   \label{eq:blow_k21}
   k_{21} &
   =
   \text{\footnotesize{$
   \dfrac{((n+1)!)^2}{(n+2)\ma(\theta)^5\cos^{2n}\theta}
   \Biggl(
   2a_{n+2,1}\,a_{n+1,1}^2 + 3 (n+2)!\,a_{n+1,1}^2\,a_{1,2}\tan\theta
   - ((n+1)!)^2 a_{n+2,1}\tan^2\theta
   \Biggr)$}}.
  \end{split}
 \end{align}

 \begin{align}
  \begin{split}
   \label{eq:alpha}
   \xi &
   =
   \text{\footnotesize$
   \frac1{\ma(\theta)^3}
   \biggl((n+2)!\,a_{n+1,1}^2\,a_{1,2}\cos^4\theta
   - a_{n+2,1}
   \bigl(((n+1)!)^2 + a_{n+1,1}^2 \bigr)\cos^3\theta\sin\theta$}\\ 
   &\quad\text{\footnotesize$
   - 2((n+1)!)^2\,a_{n+2,1}\cos\theta\sin^3\theta
   - (n+2)((n+1)!)^3\,a_{1,2}\sin^4\theta\biggr)\sin\theta$},
  \end{split}
 \end{align}
 \begin{align}
  \begin{split}
   \label{eq:beta}
   \eta &
   =
   \text{\footnotesize$
   \dfrac1{\ma(\theta)}
   \Biggl[
   -\dfrac{a_{n+1,1}^2\,a_{2,0}\,b_2}{(n+1)!}\cos^2\theta
   +\left(a_{n+1,1}(a_{2,0}^2 - b_2^2)
   - \dfrac{a_{n+3,1}}{n+2}\right)\cos\theta\sin\theta$}\\
   &\quad\text{\footnotesize$
   +\dfrac12(n+1)!(2a_{2,0}\,b_2 + (n+3)\,a_{2,2})\sin^2\theta
   \Biggr]\cos^2\theta
   +\dfrac{a_{n+1,1}}{\ma(\theta)^3}
   \Biggl[
   \dfrac{a_{n+2,1}}{n+2}\cos^2\theta$}\\
   &\quad\text{\footnotesize$
   + 2 (n+1)!\,a_{n+2,1}\,a_{1,2}\cos\theta\sin\theta
   + (n+2)!(n+1)!\,a_{1,2}^2\sin^2\theta
   \Biggr]\cos^3\theta\sin\theta$}\\
   &\quad\text{\footnotesize$
   + \dfrac{(n+1)\,a_{n+1,1}^3}{\ma(\theta)^5}
   \Biggl[
   \left(
   \dfrac{a_{n+3,1\,}a_{n+1,1}}{(n+3)(n+2)}
   - \dfrac{a_{n+2,1}^2}{(n+2)^2}\right)\cos^5\theta$}\\
   &\quad\text{\footnotesize$
   +(n+1)((n+1)!)^2\,a_{n+1,1}^3
   \left(
   \dfrac{a_{n+1,1}\,a_{2,2}}2 - \dfrac{2(a_{n+2,1}\,a_{1,2})}{n+2}
   \right)\cos^4\theta\sin\theta$}\\
   &\quad\text{\footnotesize$
   (n+1)((n+1)!)^2\,a_{n+1,1}
   \left(
   \dfrac{2a_{n+3,1}\,a_{n+1,1}}{(n+3)(n+2)}
   - \dfrac{a_{n+2,1}^2}{(n+2)^2}
   - a_{n+1,1}^2\,a_{1,2}^2
   \right)\cos^3\theta\sin^2\theta$}\\
   &\quad\text{\footnotesize$
   - (n+1)((n+1)!)^3\,a_{n+1,1}
   \left(
   \dfrac{a_{n+2,1}\,a_{1,2}}{n+2} - a_{n+1,1}\,a_{2,2}
   \right)\cos^2\theta\sin^3\theta$}\\
   &\quad\text{\footnotesize$
   + (n+1)((n+1)!)^4
   \left(
   \dfrac{a_{n+3,1}}{(n+3)(n+2)} - a_{n+1,1}\,a_{1,2}^2
   \right)\cos\theta\sin^4\theta$}\\
   &\quad\text{\footnotesize$
   + \dfrac12(n+1)((n+1)!)^5\,a_{2,2}\sin^5\theta
   \Biggr]\cos^2\theta\sin\theta
   - \epsilon
   \dfrac{2 a_{0,3}\cos\theta\sin^3\theta}{\ma(A)}$}.
  \end{split}
 \end{align}

\end{document}